\documentclass[a4paper,11pt,reqno]{amsart}

\usepackage{amsmath}
\usepackage{amssymb}
\usepackage{amsfonts}
\usepackage{graphicx}
\usepackage{mathtools}

\usepackage[colorlinks]{hyperref}

\renewcommand\eqref[1]{(\ref{#1})} 
\graphicspath{ {images/} }
\title[Hardy and Rellich identities]{Hardy and Rellich identities and inequalities for Baouendi-Grushin operators via spherical vector fields}

\author[D. Ganguly]{Debdip Ganguly}
	\address[D. Ganguly]{Theoretical Statistics and Mathematics Unit, Indian Statistical Institute, Delhi Centre, S. J. Sansanwal Marg, New Delhi, Delhi 110016, India.
	}
	\email{debdip@isid.ac.in}
\author[K. Jotsaroop]{K. Jotsaroop}
	\address[K. Jotsaroop]{Department of Mathematics, Indian Institute of Science Education and Research Mohali, Mohali, Punjab 140306, India.
	}
	\email{jotsaroop@iisermohali.ac.in}

\author[P. Roychowdhury]{Prasun Roychowdhury}
	\address[P. Roychowdhury]{Department of Mathematics, Indian Institute of Technology Hyderabad, Kandi, Sangareddy, Telangana, 502285, India
    } 
	\email{prasunrc@math.iith.ac.in}	


\date{\today}

\theoremstyle{plain}
\newtheorem{theorem}{Theorem}[section]
\newtheorem{proposition}{Proposition}[section]
\newtheorem{lemma}{Lemma}[section]
\newtheorem{corollary}{Corollary}[section]
\newtheorem{remark}{Remark}[section]
\newtheorem{definition}{Definition}[section]

\numberwithin{equation}{section} \allowdisplaybreaks

\usepackage[text={6in,8.6in},centering]{geometry}
\newcommand{\rn}{\mathbb{R}^n}
\newcommand{\re}{\mathbb{R}}
\newcommand{\rno}{\mathbb{R}^{n+1}}
\newcommand{\gradg}{\nabla_G}
\newcommand{\rgradg}{\nabla_{\varrho,G}}
\newcommand{\lapg}{\mathcal{L}_G}
\newcommand{\rlapg}{\mathcal{L}_{\varrho,G}}
\newcommand{\dx}{\:{\rm d}x}
\newcommand{\dy}{\:{\rm d}y}
\newcommand{\dsn}{\:{\rm d}\Omega}
\newcommand{\dt}{\:{\rm d}t}
\newcommand{\dr}{\:{\rm d}\varrho}
\newcommand{\dph}{\:{\rm d}\phi}
\newcommand{\dw}{\:{\rm d}w}
\newcommand{\sm}{\sum_{k=0}^\infty}
\newcommand{\norm}{\bigg|\bigg|}

\newcommand{\izfr}{\int_{0}^R}

\begin{document}
	\begin{abstract}
		For Baouendi-Grushin vector fields, we prove Hardy, Hardy-Rellich, and Rellich identities and inequalities with sharp constants. Our explicit remainder terms significantly improve than those found in the literature. Our arguments are built on abstract Hardy-Rellich identities involving the Bessel pair along with the use of spherical harmonics developed by Garofalo-Shen \cite{GS}. Furthermore, in the spirit of Bez-Machihara-Ozawa  \cite{bmo}, we construct \it spherical vector fields \rm corresponding to the Baouendi-Grushin vector fields and prove identities that, in turn, establish optimal Rellich identities, by comparing the Baouendi-Grushin operator with its radial and spherical components.  We give alternative proofs of Hardy identities and inequalities with enhanced Hardy constants in some subspaces of the Sobolev space, among other things. Additionally, we compute the deficit involving the $L^2$-norm of the Baouendi-Grushin operator and its radial component with an explicit remainder term, which leads to a comparison of the Laplacian with the radial components of the Laplacian. As a consequence of the main results, new second-order Heisenberg-Pauli-Weyl uncertainty principles and Hydrogen uncertainty principles are also derived. Furthermore, we also derive certain symmetrization principles corresponding to the Baouendi-Grushin vector fields.	
	\end{abstract}
	\maketitle
	
	\section{Introduction}
	The study of functional inequalities has become a fascinating branch of analysis and PDEs in recent years. One of the central attractions of this type of study is the Hardy inequality. The basic one-dimensional weighted Hardy inequality \cite[Theorem 330]{hlp} reads as follows: let $f\in C_c^\infty(0,\infty)$ and $\alpha\in (-1,\infty)$, then with the sharp constant, there holds
	$$
	\int_0^\infty \frac{{f^\prime}^2(x)}{x^\alpha}\dx \geq \frac{(\alpha+1)^2}{4}\int_0^\infty \frac{f^2(x)}{x^{\alpha+2}}\dx.
	$$
	
	There has been a growing interest in studying functional inequalities and identities of Hardy-Rellich type for the sub-Laplacian as square sums of vector fields on $\mathbb{R}^{n+m}.$ In particular, we consider for integers $n\geq 1, \, m\geq 1$ and real number $\alpha$, the vector fields given by
	
	\begin{equation*}
		X_i=\frac{\partial}{\partial x_i}, \quad i \, =\,1, \ldots, n,   \quad Y_{j} = |x|^{\alpha}\frac{\partial }{\partial y_j}, \ j \, = \, 1, \ldots, m
	\end{equation*}	 
	and the corresponding sub-Laplacian, well-known as the \it Baouendi-Grushin operator\rm $\,$ which was introduced in a preliminary version by Baouendi \cite{boun} and Grushin \cite{Gr1, Gr2}:
	
	\begin{equation*}
		\lapg^{\alpha} \, = \, \Delta_x \, + \, |x|^{2\alpha} \Delta_y.
	\end{equation*}
	When $\alpha =0,$ the operator $\lapg^{0}$ is reduced to the standard Laplacian in $\mathbb{R}^{n + m}.$ If $\alpha$ is a positive integer, i.e., $\alpha \in \mathbb{N},$ the vector fields $X_i's$ and $Y_j's$ satisfy H\"ormander's finite rank condition \cite{Hor1}. But in the general case, H\"ormander's condition is incoherent since the vector fields are not sufficiently smooth. Incidentally, we note that a substantial body of literature explores the case of non-smooth vector fields, where the H\"ormander condition does not hold. In this context, while not aiming for completeness, we specifically highlight the seminal work of Franchi and Lanconelli \cite{BF1}. Note that $\lapg^{\alpha}$ does not belong to the class of sub-Laplacians associated with a homogeneous group. However, it is known that $\lapg^{\alpha}$ is a subelliptic operator for $\alpha>0$ (see \cite{BF1, G1}).  The analysis of the Baouendi-Grushin operator $\lapg^{\alpha}$ is subtle, and in particular, we shall focus on the case $\alpha =1$, $n\geq 2$, and $m=1$. We will now simply refer to the \it Baouendi-Grushin \rm vector field as the ``Grushin vector" field. Here, the terminology refers to the vector fields underlying the Baouendi-Grushin operator, while the geometric framework associated with them was developed further in the work of Franchi and Lanconelli \cite{BF1} and will play a crucial role in our analysis.
	
	\medskip
	Let us first rewrite the vector fields and the Grushin operator for $\alpha =1$ and $m=1$ explicitly. Let $(x,t)\in \rn\times \re$ with $n\geq 2$. For $i=1,\ldots , n$ and consider the vector fields
	\begin{align*}
		X_i=\frac{\partial}{\partial x_i}, \quad Y = |x|\frac{\partial }{\partial t},
	\end{align*} 
	where $|x|=\sqrt{x_1^2+\cdots+x_n^2}$. Then $\rno$ equipped with the above vector fields is called the Grushin space. The Grushin gradient is defined by 
	\begin{align*}
		\gradg = (X_1,\ldots,X_n, Y) = (\nabla_x, |x|\partial_t),
	\end{align*}
	where $\nabla_x$ is the standard Euclidean gradient on $\mathbb{R}^n$. Now, recall the Grushin operator  defined by
	\begin{align*}
		\lapg=\Delta_x+|x|^2\frac{\partial^2}{\partial t^2},
	\end{align*}
	where $\Delta_x$ is the standard Laplacian on $\mathbb{R}^n$.
	Also note that, for some vector field $X$, if $\text{div}_GX=\gradg\cdot X$, then $\lapg=\text{div}_G\cdot\gradg$.
	
	Let $\delta_{\lambda}(x,t):=(\lambda x, \lambda^2 t)$ for $\lambda>0$. We can compute that $\lapg(f\circ\delta_{\lambda})=\lambda^2 \left(\lapg f\right)\circ\delta_{\lambda}$. Note that $|\delta_{\lambda}U|=\lambda^{Q}|U|$ for any $U$ Lebesgue measurable set in $\mathbb{R}^{n+1}$ where $Q=n+2$. We define on Grushin space $\rno$, a distance function for each $(x,t)\in\rno$  by
	\begin{align*}
		\varrho(x,t)=\big(|x|^4+4t^2\big)^{\frac{1}{4}}
	\end{align*}
	compatible with the dilation $\delta_{\lambda}$ defined above. From the above definition of $\varrho$, it satisfies the following: $(x,t)\mapsto\varrho(x,t)$ is $C^{\infty}(\rno\setminus\{o\}),$ positive definite, $\varrho(\delta_{\lambda}(x,t))=\lambda \varrho(x,t)$ for $\lambda>0,$ and $o=(\Vec{0},0).$ In fact, $\varrho$ is deeply connected to the fundamental solution of $\lapg$ (see,  \cite[Proposition~2.1]{G1}) and we will call $\varrho$ as a homogeneous norm on $\rno$.
	
	This paper's major goal is to prove the Hardy and Hardy-Rellich inequalities for the Grushin operator by a method that fully addresses the identity framework. The second significant point we discuss is finding the radial and spherical derivatives of the operator $\lapg$ and obtaining an identity incorporating the radial $L^2$-norm and the spherical contribution of the Grushin operator.
	
	\medskip
	
	\subsection{State of the art and main novelties} Let us discuss the current state of the art of the problems in the hypoelliptic settings briefly. First, begin with the Hardy inequality for the Baouendi-Grushin operator established by Garofalo \cite{G1}: Let $x \in \mathbb{R}^n,$ $y \in \mathbb{R}^m,$ $\alpha >0,$ with $n,m \geq 1,$ and $u\in C_c^\infty(\rno \setminus \{o\})$, then with the sharp constant there holds
	\begin{multline}\label{hardy}
			\int_{\mathbb{R}^{n+m}} \left( |\nabla_x u|^2\; + \; |x|^{2 \alpha} \; |\nabla_y u|^2\right)\dx \dy  \\ \geq \frac{(Q-2)^2}{4}\int_{\mathbb{R}^{n+m}} \left(\frac{|x|^{2 \alpha}}{|x|^{2 + 2 \alpha} + (1 + \alpha^2)^2 |y|^2} \right)|u|^2 \dx\dy,	
	\end{multline}
	where $Q = n+ (1 + \alpha)m$ is the homogeneous dimension of the Grushin space. Much research has been done on the Hardy inequalities and the best way to improve them is to add a positive remainder term to the inequality's right side \eqref{hardy}. In this regard, Kombe in \cite{IK1} proved an improved weighted Hardy type inequalities for
	Baouendi-Grushin vector fields. There are many results related to the improvement of Hardy inequality for Grushin vector fields in different directions, in particular, for the generalised class of sub-elliptic operators, and weighted $L^p$-versions of \eqref{hardy} have been proved in  \cite{AS} and \cite{AL1} respectively.  We also refer to the works of \cite{DQN, IK1, KY,  SY, YSK, ZHD} and references therein for other relevant improvements for Grushin vector fields. This is by no means an exhaustive list. It is worth mentioning that Laptev-Ruzhansky-Yessirkegenov \cite[Theorem~2.5]{laptev} were the first to attempt to enhance \eqref{hardy} by replacing the gradient in the $x$ variable with a radial contribution only.
	Their results states as follows: let $x \in \mathbb{R}^n,$ $y \in \mathbb{R}^m,$ $\alpha >0,$ with $n,m \geq 1$, $Q=n+(1+\alpha)m$,  and $u\in C_c^\infty(\rno \setminus \{o\})$, then with sharp constant there holds
		\begin{multline}\label{hardy-imp}
			\int_{\mathbb{R}^{n+m}} \left( \left|\frac{{\rm d}}{{\rm d}|x|} u\right|^2\; + \; |x|^{2 \alpha} \; |\nabla_y u|^2\right)\dx \dy  \\ \geq \frac{(Q-2)^2}{4}\int_{\mathbb{R}^{n+m}} \left(\frac{|x|^{2 \alpha}}{|x|^{2 + 2 \alpha} + (1 + \alpha^2)^2 |y|^2} \right)|u|^2 \dx\dy.
		\end{multline}
		We note that since we can estimate pointwise $\left|\frac{{\rm d}}{{\rm d}|x|} u\right| \leq |\nabla_x u|$, so \eqref{hardy-imp} gives an immediate improvement of \eqref{hardy}. We improve \eqref{hardy-imp} further by replacing the l.h.s of it with only radial derivatives (see Corollary~\ref{wg-hardy}).
	
	The aforementioned inequality has been demonstrated first within the context of Carnot groups, namely inside the Heisenberg group. We refer to the seminal work of Garofalo-Lanconelli \cite{GL}, where the following variant of the Hardy inequality on the Heisenberg group $\mathbb{H}^n$ was derived 
	\begin{equation*}
		\int_{\mathbb{H}^n}|\nabla_{\mathbb{H}^n} u|^2\, {\rm d} \xi \geq n^2
		\int_{\mathbb{H}^n} \frac{|u|^2| \psi|^2}{\varrho^2}\, {\rm d} \xi,
	\end{equation*}
	where $\varrho$ and $\psi$ are, respectively, a suitable distance from the origin and a weight function with $0 < \psi < 1$ and the corresponding homogeneous dimension $Q = 2n + 2$. Moreover, depending on the weight function $\frac{\psi^2}{\varrho^2},$ several other variants of the Hardy inequalities on $\mathbb{H}^n$ were proved in \cite{AIK, AL2, LY, NZW}. Later, Ruzhansky and Suragan extended the above Hardy inequalities to homogeneous groups, for a detailed description of these results, we refer the interested reader to \cite{RS1} (also see \cite{GK1}).   It is important to note that Ruzhansky and Suragan's most recent finding \cite{RS2} has improved the inequality mentioned above in the framework of equality, and hence obtained a sharp remainder term on the right-hand side of the inequality \eqref{hardy} for the homogeneous groups (also see \cite{RDS} for related results). They have exploited the group and dilation structures of the homogeneous groups. We will use a unified strategy to propose numerous improvements to the Hardy equality (instead of inequality) using Bessel pairs and spherical harmonics in this article. Specifically, we will re-prove all known results with improved remainder terms found in the literature. We can compare our findings with those of the Carnot groups in \cite{ dlt,  fll, GK2, LAM, QHY}. 
	
	\medskip	
	
		The next topic we shall deal with here is the concern with the validity of higher-order Hardy inequalities,  namely, the Rellich inequalities which go back to \cite{Rellich} in the Euclidean case, Hardy-Rellich inequalities go back to \cite{Hardy-Rellich, beckner}, and thereafter several improvements have been made to date. The natural extension for the Grushin operator is relatively new starting from the work of Kombe \cite{IK1}. In \cite[Theorem~4.5]{IK1}, the author proved the weighted Hardy-Rellich inequality and their improvements, which state as follows  for $\alpha = 1$ and $m =1$, $n\geq 1$ (although the author proved it for any $\alpha >0$ and $m \geq 1$): $Q=n+2$, then for $u\in C_c^\infty(\rno \setminus \{o\})$ there holds
		\begin{equation}
			\int_{\rno} \varrho^{\gamma} \frac{(\lapg u)^2}{|\gradg \varrho|^2}\dx \dt\geq \frac{(Q-\gamma)^2}{4}\int_{\rno} \varrho^{\gamma}\frac{|\nabla_G u|^2}{\varrho^2}\dx\dt,
		\end{equation}
		where $ 2 < \gamma < Q$.  Subsequently,  in \cite[Theorem 2.4]{sj}, the authors provided an improved version of the Rellich inequality for the Grushin operator.  To our understanding, the following version of the Rellich inequality in the Grushin $\varrho-$gauge ball with radius $R>0$ denoted as $B_{R}(o):=\{(x,t)\in \rno\, :\, \varrho(x,t)<R\}$ with a remainder term was first published online on arXiv circa 2007 but was later published in \cite[Theorem~4.2]{IK1}: For $Q\geq 5$ and $u\in C_c^\infty(B_R(o) \setminus \{ o \})$, then there holds
	\begin{align}\label{rellich}
		\int_{B_R(o)}\frac{(\lapg u)^2}{|\gradg \varrho|^2}\dx \dt\geq \frac{Q^2(Q-4)^2}{16}\int_{B_R(o)} \frac{|u|^2|\gradg \varrho|^2}{\varrho^4}\dx\dt \; + \;  \frac{C}{R^2} \int_{B_R(o)} \dfrac{|u|^2}{\varrho^2} \; {\rm d}x \dt,
	\end{align}
	where $Q$ is the homogeneous dimension of the Grushin space and  $C$ is some positive constant. Furthermore, without the remainder term in \eqref{rellich}, the Rellich inequality with sharp constant $\frac{Q^2(Q-4)^2}{16}$ was established in \cite[Theorem 2.4]{sj} in the full space. Moreover, in the article \cite[Theorem~2.2]{sj}, a Hardy-Rellich-type inequality was also derived, which states as follows: for $Q\geq 8$ and $u\in C_c^\infty(\rno \setminus \{o\})$, then with a sharp constant there holds
	\begin{align}\label{hardy_rellich-hr-ineq}
		\int_{\rno}\frac{(\lapg u)^2}{|\gradg \varrho|^2}\dx \dt\geq \frac{Q^2}{4}\int_{\rno} \frac{|\gradg u|^2}{\varrho^2}\dx\dt.
	\end{align}
    The above Hardy-Rellich inequality \eqref{hardy_rellich-hr-ineq}, with sharp constant $\frac{Q^2}{4}$ with homogeneous dimension $Q\geq 3$ was later established in \cite[Theorem~4.1]{KY1}. It is important to notice that in the lower Euclidean dimensional case the best Hardy-Rellich constant is $\frac{25}{36}$ and $3$ with Euclidean dimensions $3$ and $4$, respectively (see \cite[Theorem~1.1]{caza}) and this significantly differs in the Grushin setting. Driven by the aforementioned outcomes acquired within the framework of the sub-Laplacian, we will present many noteworthy enhancements of Hardy-Rellich and Rellich-type inequalities for the degenerate Grushin vector fields here. To demonstrate identities and inequalities, we will critically utilise the spherical harmonics for the Grushin operator in the present work. To the best of our knowledge, this method is new and offers a novel approach to dealing with such issues. In general, higher-order Hardy inequalities are quite delicate and subtle.  We establish several novel improvements of the Hardy-Rellich inequality, and we prove an identity for the Grushin vector fields with a sharp constant and with a precise remainder term. 
	
	\medskip
	
		In the  sub-Laplacian setting  L.~D'Ambrosio in \cite{AL2}, proved Rellich inequality on the Heisenberg group of homogeneous  dimension $Q = 2n + 2,$ for $n > 2$ with constant $n^2 (n-2)^2$. Moreover, the inequality is valid for $n > 1$ with $b: =  (n^2-1)^2$ for all the radial functions in  $C_{c}^{\infty}(\mathbb{H}^n),$ and the constant is sharp.
		Attached to the Rellich inequality proven by L.~D'Ambrosio, there is an analogous result with a slightly different form of Rellich inequality established by Y.~Qiaohua in \cite{QHY1}. In the context of the other sub-Laplacian settings, namely, in the Carnot group, Rellich-type inequalities and their improvement have been quite extensively studied in the recent past, see e.g., \cite{GKY1, RY, JS, XID} and references therein.
	
	\medskip

	We will now quickly go over this article's primary goal. Specifically, the main contributions and novelties of the present article are as follows :	
	
	\begin{itemize}	
		
		\item  In Theorem~\ref{hardy-rell-iden}, we first establish an abstract Hardy identity involving a Bessel pair. This identity can be compared with \cite[Theorem~1.1]{dlt} in the homogeneous groups setting. In fact, for the Grushin vector fields, this identity generalises all of the earlier, well-known Hardy inequalities. Additionally, we replicate every known result for the weighted version of the Hardy inequality by selecting certain Bessel pairings. Second, we establish Hardy inequalities for specific subspaces of the smooth, compactly supported functions by tweaking the proof of the abstract identity. Our primary goal is to develop a generic framework for establishing the Rellich identity, and we leverage the Hardy identity we establish in Theorem~\ref{hardy-rell-iden} to achieve this.  This could be considered the initial step in establishing the Hardy-Rellich identity.
		
		\medskip 
		
		\item We demonstrate an abstract form of the Hardy-Rellich identity, or more broadly, an inequality, by utilising the aforementioned abstract Hardy identity and the Grushin operator's structure. Given that the argument makes use of spherical harmonics, which Garofalo-Shen \cite{GS} designed for the Grushin operator, it appears to be rather delicate. This is the first use of spherical harmonics in the proof of the Hardy-Rellich inequality for the Grushin operator that we are aware of. Our obtained results include explicit remainder terms and are crisp. Rellich-type inequality, however, has been inferred for hypoelliptic operators, such as the Baouendi-Grushin operator, using the Bessel pair in  \cite{RS}.
		
		\medskip 
		
		\item Our ultimate objective is to calculate the radial and spherical contributions of the Grushin operator independently by using an identity involving the radial and spherical contributions of the Grushin operator's $L^2$-norm. We establish many identities that are comparable to the findings for Euclidean vector fields found in  \cite{bmo-snpde, bmo, mow-adv, mow}. Using a more geometric approach, we prove an identity for the Rellich situation, even for non-radial functions, as opposed to an inequality. See Theorem \ref{rellichspherical}.
		
	\end{itemize}
	
	\medskip
	
	
	\subsection{Main Results} We lay down our primary theorems. We establish the Hardy identity involving Bessel pairs in the following theorem. We employ the spherical harmonics for the Grushin vector fields in the proof. We next re-prove various well-known Hardy inequalities for the Grushin operator in the literature with much-improved remainder terms by inserting specific Bessel pairs. This reads as follows.

	\medskip
	
	\begin{theorem}\label{hardy-rell-iden}
		Let $G$ be the Grushin space with dimension $Q$ with $Q\geq 4$ and $B_R(o)$ denotes the $\varrho-$gauge ball with radius $R$ for $0<R\leq \infty$. If $(V, W)$ is a $Q$-dimensional Bessel pair on $(0,R)$ with positive solution $f$, then for $u\in C_c^\infty(B_R(o)\setminus\{o\})$ there holds
		\begin{align*}
			\int_{B_R(o)} V(\varrho)|\gradg u|^2\dx\dt&-\int_{B_R(o)}W(\varrho)u^2|\gradg \varrho|^2\dx\dt\\&=\int_{B_R(o)}V(\varrho)\bigg|\gradg\bigg(\frac{u}{f(\varrho)}\bigg)\bigg|^2f(\varrho)^2\dx\dt,
		\end{align*}
		and
		\begin{align*}
			\int_{B_R(o)} V(\varrho)|\rgradg u|^2\dx\dt&-\int_{B_R(o)}W(\varrho)u^2|\gradg \varrho|^2\dx\dt\\&=\int_{B_R(o)}V(\varrho)\bigg|\rgradg\bigg(\frac{u}{f(\varrho)}\bigg)\bigg|^2f(\varrho)^2\dx\dt.
		\end{align*}
	\end{theorem}
	
	
	\begin{remark}
		{\rm 
			
			An independent proof of \cite[Theorem 3.2]{fll}, which took into account more general Hardy identities on the Carnot groups, is provided by Theorem~\ref{hardy-rell-iden}. Nevertheless, their methodology is dependent on certain factorization lemmas. While our method is quite flexible in applying the above theorem to specific subspaces $\mathcal{S}_j$ (see Section~\ref{sec-prem} for the definition of $\mathcal{S}_j$), where the Hardy identities with an improved constant are derived (see Theorem~\ref{sub-hardy} below), it necessitates the presence of spherical harmonics for the operator and a symmetric structure of the ambient space.
		}
	\end{remark}
	
	Next, we want to apply Theorem~\ref{hardy-rell-iden} to obtain the identities and inequalities of Hardy-Rellich-type. Handling the second-order identities indeed requires considerable care. After breaking down the Grushin operator on the radial function, we obtain a Rellich identity involving Bessel pairs (for radial functions) by applying the integration by parts formula found in the previously mentioned Hardy identities on radial functions. The outcome is as follows:
	
	\begin{theorem}\label{r-rell}
		Let $G$ be the Grushin space with dimension $Q$ with $Q\geq 4$ and $B_R(o)$ denotes the $\varrho-$gauge ball with radius $R$ for $0<R\leq \infty$. Let $\left( V,  W\right)$ be a $Q$-dimensional Bessel pair on $(0, R)$ with positive solution $f$ on $(0, R)$. Then for all radial function $u \in C_{c}^{\infty}\left(B_{R}(o) \backslash\{o\}\right)$ there holds:
		\begin{align*}
			\int_{B_{R}(o)} V(\varrho) \frac{(\lapg u)^{2}}{|\gradg \varrho|^{2}} \dx\dt & =\int_{B_{R}(o)} W(\varrho)|\gradg u|^{2}\dx\dt \\
			& +(Q-1) \int_{B_{R}(o)}\left(\frac{V(\varrho)}{\varrho^{2}}-\frac{V_{\varrho}(\varrho)}{\varrho}\right)\left|\gradg u\right|^2 \dx\dt \\
			& +\int_{B_{R}(o)} V(\varrho)\left|\nabla_{G}\left(\frac{u_\varrho}{f(\varrho)}\right)\right|^{2} f^{2}(\varrho)\dx\dt,
		\end{align*}
		where the $\varrho$ in the suffix means the derivative with respect to the radial part $\varrho$.
	\end{theorem}
	
	\begin{remark}
		{\rm It is important to note that the aforementioned theorem refers to an identity rather than an inequality. We may now deduce Hardy-Rellich identities with sharp constants and precise remainder terms by selecting appropriate Bessel pairs (see Corollary~\ref{cor-hr-r-rad}). In the Euclidean setting, the aforementioned can also be compared with \cite[Theorem 1.5]{dll}. }
	\end{remark}
	
	We now demonstrate the Hardy-Rellich inequality for functions, which need not be radial. We employ the spherical harmonics to address the non-radial component of the Grushin operator in order to extend Theorem~\ref{r-rell} to the non-radial function. Restricted to the Grushin sphere, the proof contains a few interesting stages that can be handled using the spectral information of the Grushin operator.  This brings us to the subsequent theorem:
	
	\begin{theorem}\label{nr-rell}
		Let $G$ be the Grushin space with dimension $Q$ with $Q\geq 4$ and $B_R(o)$ denotes the $\varrho-$gauge ball with radius $R$ for $0<R\leq \infty$.	Let $\left( V,  W\right)$ be a $Q$-dimensional Bessel pair on $(0, R)$ with positive solution $f$ on $(0, R)$ with $V\geq 0$. Then for all $u \in C_{c}^{\infty}\left(B_{R}(o) \backslash\{o\}\right)$ there holds:
		\begin{align}\label{nr-rell-eqn}
			\int_{B_{R}(o)} V(\varrho) \frac{(\lapg u)^{2}}{|\gradg \varrho|^{2}} \dx\dt & \geq\int_{B_{R}(o)} W(\varrho)|\gradg u|^{2}\dx\dt \nonumber \\
			& +(Q-1) \int_{B_{R}(o)}\left(\frac{V(\varrho)}{\varrho^{2}}-\frac{V_{\varrho}(\varrho)}{\varrho}\right)\left|\gradg u\right|^2 \dx\dt \nonumber \\
			& +\int_{B_{R}(o)} V(\varrho)\left|\nabla_{G}\left(\frac{u_\varrho}{f(\varrho)}\right)\right|^{2} f^{2}(\varrho)\dx\dt,
		\end{align}
		where the $\varrho$ in the subscript means the derivative with respect to the radial part $\varrho$ and with the assumption
		\begin{align}\label{cond}
			(Q-5) \frac{V(\varrho)}{\varrho^{2}}+3 \frac{V_{\varrho}(\varrho)}{\varrho}-V_{\varrho\varrho}(\varrho) \geq 0 .
		\end{align}	
	\end{theorem}
	
	\begin{remark}
		{\rm	
			The assumption  \eqref{cond} is not too restrictive. In fact, in the seminal work of \cite[Theorem 3.1-3.3]{GM}, in the Euclidean space, a similar assumption was introduced involving both $V$ and $W,$ and therefore \eqref{cond} is less restrictive compared to \cite{GM}.  Our assumption provides a wide class of Bessel pairs for which \eqref{cond} is satisfied.  However, such an assumption, which only depends on $V$, was first established in \cite[Theorem 3.2]{BGR-21} on the Hyperbolic space. More recently, there has been a new development in the Euclidean setting in \cite{dll}, where the authors refined the assumption of \cite{GM} and proved several interesting consequences of their results. 
		}		
	\end{remark}	
	
	\begin{remark}\label{rem1}
		{\rm
			We note that in Theorem~\ref{nr-rell}, when we move from radial to non-radial functions, we break equality and ultimately prove an inequality in \eqref{nr-rell-eqn}. It is, therefore, reasonable to inquire as to whether we might demonstrate equality in \eqref{nr-rell-eqn} for a specific selection of $V$ and $W$.
		}		
	\end{remark}	
	
	As an immediate consequence, we will see the Hardy-Rellich and Rellich inequalities with non-negative remainder terms. We obtain the following corollaries. 
	\begin{corollary}\label{cor-hr-r}
		Let $Q\geq 5$ and $u\in C_c^\infty(\rno \setminus \{o\})$. Then there holds
		\begin{itemize}
			\item[(a)] $(\text{Hardy-Rellich inequality})$
			\begin{align*}
				\int_{\rno}\frac{(\lapg u)^{2}}{|\gradg \varrho|^{2}} \dx\dt & \geq \frac{Q^2}{4} \int_{\rno} \frac{|\gradg u|^{2}}{\varrho^2}\dx\dt \nonumber \\
				&  +\int_{\rno}\varrho^{2-Q} \left|\nabla_{G}\left(u_\varrho \varrho^{\frac{Q-2}{2}}\right)\right|^{2}  \dx\dt.
			\end{align*}
			\item[(b)] $(\text{Rellich inequality})$
			\begin{align*}
				\int_{\rno}\frac{(\lapg u)^2}{|\gradg \varrho|^2}\dx \dt&\geq\frac{Q^2(Q-4)^2}{16}\int_{\rno} \frac{|u|^2|\gradg \varrho|^2}{\varrho^4}\dx\dt\\& + \frac{Q^2}{4}\int_{\rno}\varrho^{2-Q}\left|\gradg\left(u\varrho^{\frac{Q-4}{2}}\right)\right|^2\dx\dt\\
				&  +\int_{\rno} \varrho^{2-Q}\left|\nabla_{G}\left(u_\varrho \varrho^{\frac{Q-2}{2}}\right)\right|^{2} \dx\dt.
			\end{align*}
		\end{itemize}	
	\end{corollary}
	
	Let us rewrite the result for the radial functions. Instead of an inequality, we will have an identity.
	\begin{corollary}\label{cor-hr-r-rad}
		Let $Q\geq 4$. Then for all radial functions $u\in C_c^\infty(\rno \setminus \{o\})$, there holds
		\begin{itemize}
			\item[(a)] $(\text{Hardy-Rellich identity})$
			\begin{align*}
				\int_{\rno}\frac{(\lapg u)^{2}}{|\gradg \varrho|^{2}} \dx\dt & = \frac{Q^2}{4} \int_{\rno} \frac{|\gradg u|^{2}}{\varrho^2}\dx\dt \nonumber \\
				&  +\int_{\rno}\varrho^{2-Q} \left|\nabla_{G}\left(u_\varrho \varrho^{\frac{Q-2}{2}}\right)\right|^{2}  \dx\dt.
			\end{align*}
			\item[(b)] $(\text{Rellich identity})$
			\begin{align*}
				\int_{\rno}\frac{(\lapg u)^2}{|\gradg \varrho|^2}\dx \dt&=\frac{Q^2(Q-4)^2}{16}\int_{\rno} \frac{|u|^2|\gradg \varrho|^2}{\varrho^4}\dx\dt\\& + \frac{Q^2}{4}\int_{\rno}\varrho^{2-Q}\left|\gradg\left(u\varrho^{\frac{Q-4}{2}}\right)\right|^2\dx\dt\\
				&  +\int_{\rno} \varrho^{2-Q}\left|\nabla_{G}\left(u_\varrho \varrho^{\frac{Q-2}{2}}\right)\right|^{2} \dx\dt.
			\end{align*}
		\end{itemize}
	\end{corollary}
	
The above results extensively make use of the notion of spherical harmonics
developed for $\lapg$ by Garofalo and Shen in the case $n\geq 2$, $m=1$, and
$\alpha=1$, together with the explicit form of $\varrho$. In \cite{HRL}, spherical harmonics were developed for $\lapg^{\alpha}$ with $\alpha\in\mathbb{N}$, $n\geq 2$, and $m=1$. Very recently, in \cite[Section~3]{bie-lian}, spherical harmonics for the Grushin operator were studied for $\lapg^{\alpha}$ with $\alpha\in\mathbb{N}$, $n\geq 2$, and $m\geq 2$. The methods developed in the present paper can, in principle, be extended to these settings, provided the corresponding spherical decomposition and spectral information are taken into account. For the sake of clarity, we state and prove our results in the present article for the case $\alpha=1=m$ and $n\geq 2$.
	
	\medskip
	
	Given Remark~\ref{rem1}, we shall define \it spherical vector fields \rm corresponding to the Grushin vector fields. 
	
	
	\begin{definition}
		{\rm
			Let for $j = 1, 2, \ldots, n$ we define 
			$$
			L_j=\frac{\partial}{\partial x_j}-\frac{\partial\varrho}{\partial x_j}\frac{\partial}{\partial \varrho} \quad \mbox{and} \quad 
			L_{n+1}=|x|\frac{\partial}{\partial t}-|x|\frac{\partial\varrho}{\partial t}\frac{\partial}{\partial \varrho}.
			$$
		}
	\end{definition}
	Recall that $\varrho$ is the homogeneous norm corresponding to the Grushin operator $\lapg$.
	Using the expression of $\varrho$ we can compute that $\frac{\partial\varrho}{\partial x_j}=\frac{x_j|x|^2}{\varrho^3}$ for $j=1,2,\ldots,n$ and $\frac{\partial\varrho}{\partial t}=\frac{2t}{\varrho^3}$. In the Euclidean setting, analogous vector fields were introduced by  \cite{bmo-snpde, bmo}, also (see \cite{EL}). Note that to the best of our knowledge, this is the first instance where spherical vector fields are defined in the sub-elliptic setting. In Lemma~\ref{iden-lem} below,  we shall provide all necessary (although simple) identities involving the vector fields $L_j.$ In particular, one can immediately see the following identity:
	
	$$
	\lapg = \rlapg +\sum_{j=1}^{n+1}L_j^2, 
	$$	
	where $\rlapg$ is the radial component of $\lapg$ and is defined in Section~\ref{sec-prem}. The above identity leads us to the following theorem, which compares the Grushin operator with its radial and spherical components. 
	\begin{theorem}\label{rellichspherical}
		Let $Q\geq 4$. 	For all $u\in C_c^\infty(\rno \setminus \{o\})$, we have
		\begin{align*}
			\int_{\rno}\frac{|\mathcal{L}_{G} u|^2}{|\gradg \varrho|^2}\dx \dt= \int_{\rno}\frac{|\rlapg u|^2}{|\gradg \varrho|^2}\dx \dt	+\int_{\rno}\frac{|\sum_{j=1}^{n+1}L_j^2 u|^2}{|\gradg \varrho|^2}\dx \dt\\
			+\frac{Q(Q-4)}{2}\sum_{j=1}^{n+1}\int_{\rno}\frac{|L_j u|^2}{\varrho^2}\dx\dt+2\sum_{j=1}^{n+1}\int_{\rno}\varrho^{2-Q}\left|\frac{\partial}{\partial \varrho}\left(L_ju\varrho^{\frac{Q-2}{2}}\right)\right|^2\dx\dt.\end{align*}
	\end{theorem}
	
	\medskip
	
	\begin{remark}
		{\rm It is noteworthy that classical Hardy-Rellich and Rellich identities for functions (not necessarily radial) are provided by Theorem~\ref{rellichspherical}. We may get the required identities with the exact remainder term and sharp constant by using Theorem~\ref{r-rell} on $\rlapg$. One might compare the aforementioned theorem with \cite[Theorem~1.2]{mow}. }		
		
	\end{remark}	
	
	\begin{remark}
		{\rm
			Later in Theorem~\ref{rellichprojection}, we compute the deficit, namely the $L^2$-norm of the Grushin operator and its radial components with spherical harmonics. In fact, in Theorem~\ref{rellichspherical}, this deficit is computed in terms of the $L^2$-norm of the spherical derivatives and with exact remainder terms. Later in Theorem~\ref{rellichprojection}, we shall compute this deficit in terms of spherical harmonics and using subsequent lemmas we shall quantitatively connect spherical derivatives with spherical harmonics.  We postpone this discussion for a while and we refer to Section~\ref{sec-sph-deri} for detailed statements and proofs. 
		}
	\end{remark}
	
	We end this section with the outline of the article below. 
	
	\medskip
	
	{\bf The organisation of the article is as follows: } 
	
	\begin{itemize}
		
		\item[Section 1:] The introduction section contains a brief background on the Hardy, Hardy-Rellich and Rellich inequalities for the Grushin operator, the main results of this article, and the approach we follow in the article. 
		
		\item[Section 2:] Contains the polar coordinates, spherical harmonics and spectral properties of the Grushin operator which were developed by Garofalo-Shen \cite{GS}. Moreover, the necessary preliminaries of the Bessel pair are also discussed. 
		
		\item[Section 3:] Contains the proof of several Hardy identities via Bessel pairs and some straightforward implications of the Hardy identities. 
		
		\item[Section 4:] This section is devoted to the proof of the abstract Hardy-Rellich identity. The proof is divided into several steps and as a consequence, Hardy-Rellich and Rellich inequalities with exact remainder terms are deduced. 
		
		\item[Section 5:] This section begins with the definition of spherical vector fields and factorizes the Grushin operator into the radial and spherical parts through a simple identity. Then the Rellich-type identity has been derived by isolating the radial contribution and the spherical contribution of the $L^2$ norm of the Grushin operator $\lapg$ (see Theorem~\ref{rellichspherical}). Furthermore, the spherical derivatives have been computed via spherical harmonics. 
		
		\item[Section 6:] This section deals with the symmetrization principle on the Grushin space. The main theme here is to prove an analogous P\'olya-Szeg\"o inequality for the second-order operator. The Hardy-Rellich-type identity is crucially used to prove Theorem~\ref{symmrellich}.
		
		\item[Section 7:] This section deals with applications of our theorems. The second-order uncertainty principle (USP), in particular, the Heisenberg and Hydrogen USP has been derived.  Further, employing certain scaling arguments second-order Caffarelli-Kohn-Nirenberg inequalities with sharp constants have been derived from the USP. 
		
	\end{itemize}

    \section*{Acknowledgments} 
The research of D.~Ganguly is partially supported by the ANRF MATRICS Research Grant (MTR/2023/000331) and ANRF-Advanced Research Grant (ANRF/ARG/2025/0005\\72/MS). P.~Roychowdhury is supported by the Indian Institute of Technology Hyderabad through the Seed Grant under research grant No. SG/IITH/F391/2026-27/SG-230. The authors are also grateful to Dr. Riju Basak and Dr. Suman Mukherjee for their careful reading of the manuscript and valuable comments.

\medspace

	\section{Preliminaries and functional analysis set-up}\label{sec-prem}
	
	This section is devoted to the introduction of the Grushin space, in particular, the setting up of the polar coordinates, spherical harmonics, and spectral properties. In the end, we shall also define the Bessel pair and the corresponding ordinary differential equations. 
	
	\medskip	
	
	{\bf Polar coordinates.} First, let us introduce the polar coordinate structure of this space. We can write $(x,t)=(\varrho,\phi,\theta_1,\theta_2,\ldots,\theta_{n-1})$ as below	
	\begin{align*}
		&x_1=\varrho \,\sin^{1/2}\phi\, \sin\theta_1\ldots\sin \theta_{n-1},\\
		&x_2=\varrho \,\sin^{1/2}\phi\, \sin\theta_1\ldots\cos \theta_{n-1},\\
		&\cdots\\
		&x_n=\varrho \,\sin^{1/2}\phi\, \cos\theta_1,\\
		&t=\frac{\varrho^2}{2}\cos \phi,
	\end{align*}
	where
	$0<\phi<\pi$, $0<\theta_i<\pi$ for $i=1,2,\ldots,n-2$ and $0<\theta_{n-1}<2\pi$. Here we are excluding the endpoints as those are measure zero sets. Now denoting $r=|x|$,  we have
	\begin{align*}
		r=|x|=\sqrt{x_1^2+\cdots+x_n^2}=\varrho\sin^{1/2}\phi \implies \sin\phi = \frac{|x|^2}{\varrho^2(x,t)}.
	\end{align*}
	We define $\psi(x,t)=|\gradg \varrho(x,t)|^2$.
	Now this will give 
	\begin{align*}
		|\gradg \varrho(x,t)|^2=|\nabla_x\varrho|^2+|x|^2(\partial_t \varrho)^2=\frac{|x|^2}{\varrho^2(x,t)}=\psi(x,t)=\sin\phi = \frac{r^2}{\varrho^2},
	\end{align*}	
	and 
	\begin{align*}
		\lapg\varrho (x,t) = (Q-1)\frac{|x|^2}{\varrho^3(x,t)}=(Q-1)\frac{\psi(x,t)}{\varrho(x,t)}.
	\end{align*}
	
	Using polar coordinates, we obtain
	\begin{align*}
		\dx=r^{n-1}{\rm d}r\,{\rm d}w, \text{ and }
		\dx\dt=\frac{1}{2}\varrho^{n+1}(\sin\phi)^{\frac{n-2}{2}}\dr\dph\dw,	
	\end{align*}
	where $\dw$ is the spherical measure on the standard round sphere $\mathbb{S}^{n-1}$ in the Euclidean space $\mathbb{R}^n$. Now consider the $\varrho$-sphere on the Grushin space as follows
	\begin{align*}
		\Omega=\biggl\{(x,t)\in \rno : \varrho(x,t)=1\biggr\}.
	\end{align*}
	Moreover, using the standard argument there holds the following polar integration. Let $u$ be some integrable function, and then we have
	\begin{align*}
		\int_{\rno}u(x,t) \dx \dt&=\int_{\Omega}\int_{0}^{\infty}u(\varrho,\phi,w)\frac{1}{2}\varrho^{n+1}(\sin\phi)^{\frac{n-2}{2}}\dr\dph\dw\\&=\int_{\Omega}\int_{0}^{\infty}u(\varrho,\sigma)\frac{\varrho^{n+1}}{2\psi}\dr\dsn,
	\end{align*}
	where $${\rm d}\Omega=(\sin\phi)^{\frac{n}{2}}\dph\dw, \text{ and } \sigma=(\phi,w), \text{ with }w\in \mathbb{S}^{n-1}.$$
	
	
	\medskip	
	
	{\bf Operators and their radial part.} Now we will write all the operators in terms of the polar coordinate structure. The Grushin gradient in polar coordinates may be defined as follows
	\begin{align*}
		\gradg=\psi^{1/2}\bigg(\partial_\varrho, \frac{2}{\varrho}\partial_\phi,\frac{1}{\varrho\sin\phi}\nabla_w\bigg),
	\end{align*}
	where $(x,t)=(\varrho,\phi,w)$, $w\in \mathbb{S}^{n-1}$, $\nabla_w$ is the gradient on $\mathbb{S}^{n-1}$. Also, it can be written as follows
	\begin{align*}
		\gradg=\psi^{1/2}\bigg(\partial_\varrho, \frac{2}{\varrho}\nabla_\sigma\bigg),
	\end{align*}
	where $\nabla_\sigma$ is the gradient operator on the Grushin $\varrho$-sphere. We refer to \cite[Page. 11]{flynn} for more details.	
	In polar coordinates, the Grushin operator takes the form 
	\begin{align}\label{grushin-polar}
		\lapg=\psi\biggl\{\frac{\partial^2}{\partial \varrho^2}+\frac{Q-1}{\varrho}\frac{\partial}{\partial \varrho}+\frac{4}{\varrho^2}\mathcal{L}_\sigma\biggr\},
	\end{align}
	where $\sigma=(\phi,w)$ and $\mathcal{L}_\sigma$ is the Laplacian on Grushin $\varrho$-sphere.
	
	To this end, we shall define the radial contribution of the operators. The \it radial Grushin gradient and the radial Grushin operator \rm are defined as
	\begin{align*}
		\nabla_{\varrho,G}=\psi^{1/2}\big(\partial_\varrho,0,0\big), \quad \text{ and } \quad \mathcal{L}_{\varrho,G}=\psi\biggl\{\frac{\partial^2}{\partial \varrho^2}+\frac{Q-1}{\varrho}\frac{\partial}{\partial \varrho}\biggr\}.
	\end{align*}
	
	{\bf Spherical harmonics.} Note that, for $k=0,1,\ldots,$ we form the function $u_k(\varrho, \sigma):=\varrho^k g(\sigma)$. Then, $u_k(\varrho, \sigma)$ is a solution of $\mathcal{L}_{G}u_k = 0$ if and only if 
	\begin{align}\label{eigen-sphere}
		\mathcal{L}_\sigma g=-\frac{k(n+k)}{4}g.
	\end{align}
	Before going further, let us mention some key observations of spherical harmonics in Grushin space. Let us recall some known facts  :
	
	\begin{lemma}\cite[Lemma 2.6]{GS}
		Let $k$ be a nonnegative integer and $l \equiv k(\mbox{mod}\ 2),$ with $0 \leq l \leq k.$ Suppose that $Y_l$ is a spherical harmonic of degree $l$ on $\mathbb{S}^{n-1}$. Then 
		$$
		g_{l,k}(\phi, w) : = g_{l,k}(\sigma) = \sin^{l/2}(\phi) C_{\frac{k-l}{2}}^{\frac{l}{2} + \frac{n}{4}}(\cos \phi) Y_{l}(w) 
		$$
		satisfies \eqref{eigen-sphere}, where $C_{\frac{k-l}{2}}^{\frac{l}{2} + \frac{n}{4}}$ is a Gegenbauer polynomial. 
	\end{lemma}
	
	Now we define
	$$
	\mathcal{H}_k := \, \mbox{Span} \left\{ \sin^{l/2}(\phi) C_{\frac{k-l}{2}}^{\frac{l}{2} + \frac{n}{4}}(\cos \phi) Y_{l, j}(w) :
	\ j = 1, 2, \ldots, d_l, \quad 0 \leq l \leq k, \quad l \equiv k(\mbox{mod}\ 2) \right\},
	$$	
	where $d_l = \frac{(n+2l -2)\Gamma(n + l-2)}{\Gamma(l+1) \Gamma(n-1)}$ and $\{Y_{l,j}\}_{j=1,2,\ldots,d_l}$ is an orthonormal basis for the space of spherical harmonics of degree $l$ on $\mathbb{S}^{n-1}$.

	\begin{lemma}\cite[Lemma 2.11]{GS}
		The following direct sum decomposition holds : 
		\begin{align*}
			L^2(\Omega, {\rm d}\Omega)=\bigoplus_{k=0}^\infty\mathcal{H}_k,
		\end{align*}
		where the spaces $\mathcal{H}_k$ are mutually orthonormal in $L^2 (\Omega, {\rm d}\Omega).$
	\end{lemma}

	Given the above lemmas, we have, $u(x, t)=u(\varrho,\sigma)\in C_{c}^\infty(\rno\setminus\{o\})$, $\varrho\in({0},\infty)$ and $\sigma\in \Omega$, thus we can write
	\begin{equation*}
		u(\varrho,\sigma)=\sum_{k=0}^{\infty}d_{k}(\varrho)\Phi_k(\sigma)
	\end{equation*}
	in $L^2(\rno)$, where $\{ \Phi_k \}$ is an orthonormal system of spherical harmonics in $L^2(\Omega)$ and 
	\begin{equation*}
		d_{k}(\varrho)=\int_{\Omega}u(\varrho,\sigma)\Phi_k(\sigma) \ {\rm d}\Omega\,.
	\end{equation*} 
	A spherical harmonic $\{ \Phi_k \}$ of order $k$ satisfies $$-\mathcal{L}_{\sigma}\Phi_k=\lambda_k\Phi_k,$$
	for all $k\in\mathbb{N}\cup\{0\}$, where $\lambda_k=\frac{k(k+n)}{4}$. In this context, let us write the gradient and Grushin operators as follows :
	\begin{align*}
		|\gradg u|^2 =  \psi  \bigg(\left(\partial_{\varrho} u\right)^2  + \frac{4}{\varrho^2}  |\nabla_{\sigma}u|^2\bigg),
	\end{align*}
	
	and
	\begin{align*}
		(\lapg u)^2 & = \psi^2 \left( \partial^2_{\varrho}u + \frac{(Q-1)}{\varrho} \partial_{\varrho}u \right)^2  + \psi^2
		\frac{16}{\varrho^4} (\mathcal{L}_{\sigma} u)^2 \\
		& + 2\psi^2 \left( \partial^2_{\varrho}u + \frac{(Q-1)}{\varrho} \partial_{\varrho}u \right) \frac{4}{\varrho^2} (\mathcal{L}_{\sigma} u) u.\notag
	\end{align*}
	
	Along with this, the radial contribution of the operators will look as follows :
	\begin{align*}
		|\rgradg u|^2 =  \psi \left(\partial_{\varrho} u\right)^2	,
	\end{align*}
	and
	\begin{align*}
		(\rlapg u)^2 & = \psi^2 \left( \partial^2_{\varrho}u + \frac{(Q-1)}{\varrho} \partial_{\varrho}u \right)^2\,.
	\end{align*}
	
	\medspace
	
	We can deduce the following integration by parts formula :			
	\begin{lemma}\label{int-by-parts}
		Let $\Phi_k \in \mathcal{H}_k$ and $\Phi_j \in \mathcal{H}_j,$ then there holds  
		$$\int_{\Omega} \nabla_\sigma \Phi_k(\sigma) \nabla_\sigma \Phi_j(\sigma) \dsn =-\int_{\Omega}\mathcal{L}_{\sigma}\Phi_k(\sigma) \Phi_j(\sigma)\dsn,$$ 
		for all $k, \,j \in \mathbb{N}.$
	\end{lemma}
	\begin{proof}
		Let $f\in C^{\infty}_c(0,\infty)$ and for $k\in\mathbb{N}$, we define
		\begin{align*}
			u_k(x,t):=u_k(\varrho,\sigma)=f(\varrho)\Phi_k(\sigma) \quad \text{ for all } (x,t)\in\rno.
		\end{align*}
		Now, let us describe the explicitly known $\Phi_k(\sigma)$. Let $l\equiv m(\text{mod}\, 2)$. Then we know 
		\begin{align*}
			g_{l,m}(\phi,\omega)=(\sin\phi)^{l/2}C^{\frac{l}{2}+\frac{n}{4}}_{\frac{m-l}{2}}(\cos\phi)Y_{l}(\omega),
		\end{align*} 
		where $Y_{l}$ is an $l$ degree spherical harmonic on $\mathbb{S}^{n-1}$ w.r.t. the standard Laplacian on $\mathbb{R}^n$. In the spherical polar coordinates on the Grushin space, we have $|x|=\varrho (\sin\phi)^{1/2}$. Therefore, we can write
		\begin{align*}
			(\sin\phi)^{l/2}Y_{l}(\omega)=Y_{l}\left(\frac{x}{\varrho}\right).
		\end{align*}
		
		Note that when $\varrho$ lies in a compact subset of $(0,\infty)$, then
		\begin{align*}
			(x,t)\longmapsto \varrho(x,t)^{\gamma}, \quad \text{ for } \gamma\in \mathbb{R},
		\end{align*}
		is a smooth function of $(x,t)$. Similarly,
		\begin{align*}
			(x,t)\longmapsto C^{\frac{l}{2}+\frac{n}{4}}_{\frac{m-l}{2}}\left(\frac{2t}{\varrho(x,t)^2}\right)   \quad \text{ and } \quad (x,t)\longmapsto Y_{l}\left(\frac{x}{\varrho}\right)
		\end{align*}
		are smooth functions of $(x,t)$. As ${\Phi_{k}}'s$ are of the form $g_{l,m}$, for some $m\geq 0$ and $l\equiv m(\text{mod}\,2)$. Therefore, we can apply integration by parts in $\rno$ to $u_k$ and $u_j$.
		
		Set
		\begin{align*}
			I=\int_{\mathbb{R}^{n+1}}(\gradg u_k,\gradg u_j)\dx \dt.
		\end{align*}
		In polar coordinates, the above integral can be split into a sum of two terms involving the gradient w.r.t. radial and spherical coordinates as follows
		\begin{multline*}
			I=\int_0^{\infty}\int_{\Omega}f'(\varrho)^2\Phi_k(\sigma)\Phi_j(\sigma) \varrho^{Q-1}\dsn\dr\\+\int_0^{\infty}\int_{\Omega}f(\varrho)^2 4\varrho^{-2} \nabla_\sigma \Phi_k(\sigma) \nabla_\sigma \Phi_j(\sigma) \varrho^{Q-1}\dsn\dr.
		\end{multline*}
		On integration by parts by treating $I$ as an integral on $\rno$, we get that 
		\begin{align*}
			I=-\int_{\rno}\lapg (u_k)u_j \dx\dt.
		\end{align*}
		Again, splitting the radial and angular parts of $\lapg$, the above expression can be written as 
		\begin{multline*}
			I=-\int_0^{\infty}\int_{\Omega}\rlapg f(\varrho)f(\varrho)\Phi_k(\sigma) \Phi_j(\sigma)\varrho^{Q-1}\dsn\dr\\ -\int_0^{\infty}\int_{\Omega}f(\varrho)^2 4\varrho^{-2}\mathcal{L}_{\sigma}\Phi_k(\sigma) \Phi_j(\sigma)\varrho^{Q-1}\dsn\dr.
		\end{multline*}
		On equating the above two expressions for all $f\in C^{\infty}_c(0,\infty)$ we get that $$\int_{\Omega} \nabla_\sigma \Phi_k(\sigma) \nabla_\sigma \Phi_j(\sigma) \dsn =-\int_{\Omega}\mathcal{L}_{\sigma}\Phi_k(\sigma) \Phi_j(\sigma)\dsn$$ for all $k,j\in\mathbb{N}$. This completes the proof.
	\end{proof}	
	
	\medskip	
	
	{\bf Projection operators.} For each $k\geq 0$, we define $\dot{\mathcal{H}}_k(\rno\setminus\{o\})$ as the closed subspace spanned by spherical harmonics of order $k$ multiplied by radial functions. Any function $u(x)\in C_{c}^\infty(\rno\setminus\{o\})$ can be expressed as follows:
	\begin{equation*}
		u(x)=u(\varrho,\sigma)=\sum_{k=0}^{\infty}\mathcal{P}_k u,
	\end{equation*}
	where $\mathcal{P}_k u=d_{k}(\varrho)\Phi_k(\sigma)\in \dot{\mathcal{H}}_k(\rno\setminus\{o\})$ and $\mathcal{P}_k$ is the projection operator on $\dot{\mathcal{H}}_k(\rno\setminus\{o\}).$
	
	
	\medspace
	
	{\bf Subspace.} We have already introduced the spherical harmonics decomposition for a function $u\in C_c^\infty(\rno\setminus\{o\})$.  We will introduce a certain subspace of $C_c^\infty(\rno\setminus\{o\})$. First, consider $u\in C_c^\infty(\rno\setminus\{o\})$ and we can write $$u(x)=u(\varrho,\sigma)=\sum_{k=0}^{\infty}d_{k}(\varrho)\Phi_k(\sigma)=\sum_{k=0}^{\infty}\mathcal{P}_k(\varrho,\sigma)$$
	and now define for $j\in\mathbb{N}\cup \{0\}$ and $\Omega\subset \rno\setminus\{o\}$,
	\begin{align*}
		\mathcal{S}_j(\Omega)&=\biggl\{u\in C_c^\infty(\Omega) : \mathcal{P}_0(\varrho,\sigma)=\mathcal{P}_1(\varrho,\sigma)=\cdots=\mathcal{P}_j(\varrho,\sigma)=0\biggr\}\\
		&=\biggl\{u\in C_c^\infty(\Omega) : u(x)=u(\varrho,\sigma)=\sum_{k=j+1}^{\infty}d_{k}(\varrho)\Phi_{k}(\sigma)\biggr\}.
	\end{align*}
	For notational clarity, we want to mention that $\mathcal{S}_{-1}(\Omega)=C_c^\infty(\Omega).$
	
	\medskip

	\begin{definition}[Bessel pair]
		We say that a pair $(V,W)$ of $C^1$-functions is a Bessel pair on $(0,R)$ for some $0<R\leq \infty$  if the ordinary differential equation:
		\begin{equation*}
			(Vy')'+Wy=0
		\end{equation*} 
		admits a positive solution $f$ on the interval $(0,R)$.
	\end{definition}
	\begin{definition}
		We say that a pair $(V,W)$ of $C^1$-functions is a $Q$-dimensional Bessel pair on $(0,R)$ for some $0<R\leq \infty$  if the ordinary differential equation:
		\begin{equation*}
			(\varrho^{Q-1}Vy')'+ \varrho^{Q-1}Wy=0
		\end{equation*} 
		admits a positive solution $f$ on the interval $(0,R)$.
	\end{definition}
	
	\medspace
	
	\section{Proof of Theorem~\ref{hardy-rell-iden} and corollaries}
	
	The purpose of this section is to prove the Hardy-type identity using a Bessel pair, namely Theorem~\ref{hardy-rell-iden}. Additionally, we will prove multiple enhancements of the Hardy inequality, including the improvement of the optimality of the classical Hardy constant and the weighted version of the classical Hardy inequality. Finally, we will develop a refined Hardy-type inequality for functions in $\mathcal{S}_j$. The proof of Theorem~\ref{hardy-rell-iden} is where we start.
	
	\medskip
	
	{\bf Proof of Theorem~\ref{hardy-rell-iden}:} Let us begin with $u\in C_c^\infty(B_R(o)\setminus\{o\})$ and form the spherical harmonic decomposition of $u$ we can write
		\begin{align*}
			u(\varrho,\sigma)=\sm d_k(\varrho)\Phi_k(\sigma).
		\end{align*}	
		Next, exploiting the polar coordinate structure,  using the orthonormal properties of $\{\Phi_{k}\},$ and Lemma~\ref{int-by-parts}, we deduce
		\begin{align*}
			&\int_{B_R(o)} V(\varrho)\left|\gradg \left(\frac{u}{f(\varrho)}\right)\right|^2f(\varrho)^2\dx\dt\\
			&=\frac{1}{2}\sm \int_{\Omega}\izfr  V\bigg[ {( d_k/f)^\prime}^2\Phi_k^2+\frac{4d_k^2}{f^2\varrho^2}|\nabla_\sigma \Phi_k|^2\bigg]\varrho^{n+1}f^2\dr\dsn\\
			&=\frac{1}{2}\sm\int_{0}^R\bigg[V{d_k^\prime}^2\varrho^{n+1}-V{(d_k^2)}^\prime \frac{f^\prime}{f}\varrho^{n+1}+V\frac{{f^\prime}^2}{f^2}d_k^2\varrho^{n+1}+4\lambda_k V d_k^2\varrho^{n-1}\bigg]\dr.
		\end{align*}
		Now, using integration by parts, we deduce,
		\begin{align*}
			\izfr V{(d_k^2)}^\prime \frac{f^\prime}{f}\varrho^{n+1}\dr&=-\izfr V^\prime d_k^2 \frac{f^\prime}{f}\varrho^{n+1}\dr+\izfr V d_k^2 \frac{{f^\prime}^2}{f^2}\varrho^{n+1}\dr\\&-\izfr V d_k^2 \frac{f^{\prime\prime}}{f}\varrho^{n+1}\dr-(Q-1)\izfr V d_k^2 \frac{f^\prime}{f}\varrho^{n}\dr.
		\end{align*}
		Therefore, we have
		\begin{align*}
			&\int_{B_R(o)} V(\varrho)\left|\gradg \left(\frac{u}{f(\varrho)}\right)\right|^2f(\varrho)^2\dx\dt\\
			&=\frac{1}{2}\sm\izfr\bigg[V{d_k^\prime}^2\varrho^{n+1}+ V^\prime d_k^2 \frac{f^\prime}{f}\varrho^{n+1}+ V d_k^2 \frac{f^{\prime\prime}}{f}\varrho^{n+1}\\&+(Q-1) V d_k^2 \frac{f^\prime}{f}\varrho^{n}+4\lambda_k V d_k^2\varrho^{n-1}\bigg]\dr.
		\end{align*}
		Also, we have
		\begin{align*}
			&\int_{B_R(o)} V(\varrho)|\gradg u|^2\dx\dt\\
			&=\frac{1}{2}\sm \int_{\Omega}\izfr  V\bigg[ { d_k^\prime}^2\Phi_k^2+\frac{4 d_k^2}{\varrho^2}|\nabla_\sigma \Phi_k|^2\bigg]\varrho^{n+1}\dr\dsn\\
			&=\frac{1}{2}\sm\int_{0}^R\bigg[V{d_k^\prime}^2\varrho^{n+1}+4\lambda_k V d_k^2\varrho^{n-1}\bigg]\dr.
		\end{align*}
		Finally, after substituting and using the equation corresponding to the Bessel pair we obtain
		\begin{align*}
			&\int_{B_R(o)} V(\varrho)|\gradg u|^2\dx\dt-	\int_{B_R(o)}V(\varrho)\bigg|\gradg\bigg(\frac{u}{f(\varrho)}\bigg)\bigg|^2f(\varrho)^2\dx\dt\\
			&=\frac{1}{2}\sm\izfr\bigg[-V^\prime  f^\prime-V f^{\prime\prime}-(Q-1)\frac{V}{\varrho}  f^\prime\bigg]\frac{d_k^2}{f}\varrho^{n+1}\dr\\
			&=\frac{1}{2}\sm\izfr Wd_k^2\varrho^{n+1}\dr\\
			&=\int_{B_R(o)}W(\varrho)u^2|\gradg \varrho|^2\dx\dt.
		\end{align*}
		The second identity follows with similar steps and decomposition. $\qed$
	
	\medskip

	The above theorem leads us to the following result. 	
	\begin{corollary}\label{sym-hardy}
		Let $G$ be the Grushin space with dimension $Q$ with $Q\geq 4$ and $B_R(o)$ denotes the $\varrho-$gauge ball with radius $R$ for $0<R\leq \infty$. If $(V, W)$ is a $Q$-dimensional Bessel pair on $(0,R)$, with $V\geq 0$, then for $u\in C_c^\infty(B_R(o)\setminus\{o\})$ there holds
		\begin{align*}
			\int_{B_R(o)} V(\varrho)|\gradg u|^2\dx\dt\geq \int_{B_R(o)}W(\varrho)u^2|\gradg \varrho|^2\dx\dt,
		\end{align*}
		and
		\begin{align*}
			\int_{B_R(o)} V(\varrho)|\rgradg u|^2\dx\dt\geq \int_{B_R(o)}W(\varrho)u^2|\gradg \varrho|^2\dx\dt.
		\end{align*}
	\end{corollary}
	
	
	In the next theorem, we shall refine the above theorem for functions in $\mathcal{S}_j,$ a subspace of $C_c^{\infty}( \mathbb{R}^{n+1}\setminus \{ o\}).$ We deduce that the optimal Hardy constant depends on the parameter $j.$ Also, see 
	\cite{GR-23} for related results in Riemannian symmetric manifolds, in particular, in the hyperbolic space. 
	
	\begin{theorem}\label{sub-hardy}
		Let $G$ be the Grushin space with dimension $Q$ with $Q\geq 4$ and $B_R(o)$ denotes the $\varrho-$gauge ball with radius $R$ for $0<R\leq \infty$. If $(V, W)$ is a $Q$-dimensional Bessel pair on $(0,R)$ with positive solution $f$, then for $u\in \mathcal{S}_j(B_R(o)\setminus\{o\})$ there holds
		\begin{align*}
			\int_{B_R(o)} V(\varrho)|\gradg u|^2\dx\dt&\geq\int_{B_R(o)}W(\varrho)u^2|\gradg \varrho|^2\dx\dt\\&+(j+1)(Q+j-1)\int_{B_R(o)}\frac{V(\varrho)}{\varrho^2}u^2|\gradg \varrho|^2\dx\dt\\&+\int_{B_R(o)}V(\varrho)\bigg|\rgradg\bigg(\frac{u}{f(\varrho)}\bigg)\bigg|^2f(\varrho)^2\dx\dt.
		\end{align*}
	\end{theorem}
	\begin{proof}
		The proof is exactly similar to the proof of Theorem \ref{hardy-rell-iden}. In that proof, the term involving $\lambda_{k}$ was canceled. But in this case, as the last term is a radial gradient, the term involving $\lambda_{k}$ will not be canceled out. Then using the estimate $4\lambda_{k}\geq (j+1)(Q+j-1)$ for all $k\geq j+1$ on $\mathcal{S}_j(B_R(o)\setminus\{o\})$ the result follows.
	\end{proof}
	
	\medskip
	
	\begin{remark}
		{\rm	
			Observe that $\big(1,\frac{(Q-2)^2}{4}\frac{1}{\varrho^2}\big)$ is a $Q$-dimensional Bessel pair on $(0,\infty)$ with positive solution $\varrho^{-\frac{Q-2}{2}}$. Then, substituting this particular Bessel pair in Theorem \ref{sub-hardy}, we obtain the Hardy inequality with a constant larger than the classical Hardy constant. This reads as follows : for $u\in \mathcal{S}_j(B_R(o)\setminus\{o\})$ there holds
			\begin{align*}
				\int_{\rno} |\gradg u|^2\dx\dt&\geq\bigg[\frac{(Q-2)^2}{4}+(j+1)(Q+j-1)\bigg]\int_{\rno}\frac{|u|^2}{\varrho^2}|\gradg \varrho|^2\dx\dt.
			\end{align*}
		}		
	\end{remark}
	
	\medskip
	
	\begin{corollary}\label{wg-hardy}
		Let $G$ be the Grushin space with dimension $Q\geq 4$ and $\alpha\in\mathbb{R}$. 
		Then for all $u\in C_c^\infty(\rno\setminus\{o\})$ there holds
		\begin{align*}
			\int_{\rno} \frac{|\gradg u|^2}{\varrho^\alpha}\dx\dt&-\frac{(Q-2-\alpha)^2}{4}\int_{\rno}\frac{|u|^2|\gradg \varrho|^2}{\varrho^{\alpha+2}}\dx\dt\\&=\int_{\rno}\varrho^{2-Q}\big|\gradg\big(u\varrho^{\frac{Q-2-\alpha}{2}}\big)\big|^2\dx\dt,
		\end{align*}
		and
		\begin{align*}
			\int_{\rno} \frac{|\rgradg u|^2}{\varrho^\alpha}\dx\dt&-\frac{(Q-2-\alpha)^2}{4}\int_{\rno}\frac{|u|^2|\gradg \varrho|^2}{\varrho^{\alpha+2}}\dx\dt\\&=\int_{\rno}\varrho^{2-Q}\big|\rgradg\big(u\varrho^{\frac{Q-2-\alpha}{2}}\big)\big|^2\dx\dt.
		\end{align*}
	\end{corollary}
	\begin{proof}
		We apply Theorem \ref{hardy-rell-iden} to the $Q$-dimensional Bessel pair $\left(\frac{1}{\varrho^\alpha}, \frac{(Q-2-\alpha)^2}{4}\frac{1}{\varrho^{\alpha+2}}\right)$ on $(0,\infty)$ with $f(\varrho)=\varrho^{\frac{2-Q+\alpha}{2}}$ to get the desired results.
	\end{proof}
	\begin{remark}
		{\rm	
			The above corollary can be compared with the identity similar to \cite[Theorem 1]{mow-adv} on the Euclidean space in terms of the weighted version. Notice that we have used the spherical harmonic technique, whereas in \cite{mow-adv} a different approach was used.  Moreover, in the above identity by neglecting the non-negative remainder term, we obtain 
			\begin{align*}
				\int_{\rno} \frac{|\rgradg u|^2}{\varrho^\alpha}\dx\dt\geq \frac{(Q-2-\alpha)^2}{4}\int_{\rno}\frac{|u|^2|\gradg \varrho|^2}{\varrho^{\alpha+2}}\dx\dt,
			\end{align*}
			holds for every $u\in C_c^\infty(\rno\setminus\{o\})$. Now, by noticing that $|\gradg u|^2\geq |\nabla_{\varrho, G}u|^2$, we deduce that this weighted Hardy inequality immediately gives an improvement of the original weighted Hardy inequality on Grushin space. Also, this implies that the constant is the optimal one.
		}		
	\end{remark}

		\begin{corollary}
			Let $G$ be the Grushin space with dimension $Q$ with $Q\geq 4$ and $B_R(o)$ denotes the $\varrho-$gauge ball with radius $R>0$. Then for $u\in C_c^\infty(B_R(o)\setminus\{o\})$ there holds
			\begin{align*}
				\int_{B_R(o)}|\gradg u|^2\dx\dt&-\frac{(Q-2)^2}{4}\int_{B_R(o)}\frac{u^2|\gradg \varrho|^2}{\varrho^2}\dx\dt-\frac{z_0^2}{R^2}\int_{B_R(o)}u^2|\gradg \varrho|^2\dx\dt\\&=\int_{B_R(o)}\varrho^{2-Q}J_0^2\left(\frac{\varrho\,z_0}{R}\right)\bigg|\gradg\bigg(\varrho^{\frac{Q-2}{2}}\frac{u}{J_0(\frac{\varrho\,z_0}{R})}\bigg)\bigg|^2\dx\dt,
			\end{align*}
			and
			\begin{align*}
				\int_{B_R(o)}|\rgradg u|^2\dx\dt&-\frac{(Q-2)^2}{4}\int_{B_R(o)}\frac{u^2|\gradg \varrho|^2}{\varrho^2}\dx\dt-\frac{z_0^2}{R^2}\int_{B_R(o)}u^2|\gradg \varrho|^2\dx\dt\\&=\int_{B_R(o)}\varrho^{2-Q}J_0^2\left(\frac{\varrho\,z_0}{R}\right)\bigg|\rgradg\bigg(\varrho^{\frac{Q-2}{2}}\frac{u}{J_0(\frac{\varrho\,z_0}{R})}\bigg)\bigg|^2\dx\dt,
			\end{align*}
			where $z_0=2.4048\ldots$ is the first zero of the Bessel function $J_0(z)$.
		\end{corollary}
		\begin{proof}
			Notice that $\big(1,\frac{(Q-2)^2}{4}\frac{1}{\varrho^2}+\frac{z_0^2}{R^2}\big)$ is a $Q$-dimensional Bessel pair on $(0, R)$ for $R>0$ with positive solution $f(\varrho)=\varrho^{-\frac{Q-2}{2}}J_0(\frac{\varrho\,z_0}{R})$ and then applying this in Theorem~\ref{hardy-rell-iden} we obtain the required identities.
		\end{proof}
		\begin{remark}
			\rm The above results immediately give the improvements of the Hardy inequality of Brezis-V\'azquez type for the Grushin operator. This type of improvement for this setting was first observed in results \cite[Theorem~1.1]{YSK}, but we obtain it in the full spirit of \cite[Theorem~4.1]{bv} with explicit remainder terms. \it
		\end{remark}

	\medspace
	
	\section{Abstract Hardy-Rellich inequality: Proof of Theorem~\ref{r-rell} and Theorem~\ref{nr-rell}}
	Now we will prove the abstract Hardy-Rellich identity via the Bessel pair for radial functions. We shall first prove our theorem for radial functions, namely, Theorem~\ref{r-rell}. The proof is a straightforward application of integration by parts, and expanding the square of the Grushin operator on the radial function.  We begin the proof of Theorem~\ref{r-rell}.

	{\bf Proof of Theorem~\ref{r-rell}:} Let us assume $\psi=\sin\phi = |\gradg \varrho|^2$. Also observe that for the radial function $u$, we have $|\gradg u|^2=\psi\big(\frac{\partial u}{\partial \varrho}\big)^2$. Now, using polar coordinate decomposition, we have
	\begin{align*}
		&\int_{B_{R}(o)} V(\varrho) \frac{(\lapg u)^{2}}{|\gradg \varrho|^{2}} \dx\dt\\ &=\frac{1}{2}	\int_{\Omega}\int_{0}^{R}V(\varrho)\frac{\psi^2}{\psi}\bigg[\frac{\partial^2 u}{\partial \varrho^2}+\frac{Q-1}{\varrho}\frac{\partial u}{\partial \varrho}\bigg]^2\varrho^{n+1}(\sin\phi)^{\frac{n-2}{2}}\dr\dph\dw\\&=\frac{1}{2}	\int_{\Omega}\int_{0}^{R}V(\varrho)\bigg[\frac{\partial^2 u}{\partial \varrho^2}+\frac{Q-1}{\varrho}\frac{\partial u}{\partial \varrho}\bigg]^2\varrho^{n+1}(\sin\phi)^{\frac{n}{2}}\dr\dph\dw\\&=\frac{1}{2}	\int_{0}^{R}V(\varrho)\bigg[\frac{\partial^2 u}{\partial \varrho^2}+\frac{Q-1}{\varrho}\frac{\partial u}{\partial \varrho}\bigg]^2\varrho^{n+1}\dr\int_{\Omega}(\sin\phi)^{\frac{n}{2}}\dph\dw\\&=\frac{1}{2}|\Omega|	\int_{0}^{R}V(\varrho)\bigg[\frac{\partial^2 u}{\partial \varrho^2}+\frac{Q-1}{\varrho}\frac{\partial u}{\partial \varrho}\bigg]^2\varrho^{n+1}\dr\\&=\frac{1}{2}|\Omega| \int_{0}^{R} V(\varrho) \bigg[ \left(\frac{\partial^{2} u}{\partial \varrho^{2}}\right)^{2}+\frac{(Q-1)^{2}}{\varrho^{2}}\left(\frac{\partial u}{\partial \varrho}\right)^{2}+2\frac{(Q-1)}{\varrho} \frac{\partial^{2} u}{\partial \varrho^{2}} \frac{\partial u}{\partial \varrho}\bigg] \varrho^{n+1} \dr.
	\end{align*}
	We will deal with each integral separately. Let us define and apply Theorem \ref{hardy-rell-iden}, and we have
	\begin{align*}
		I_1&:=\frac{1}{2}|\Omega| \int_{0}^{R} V(\varrho)\left(\frac{\partial^{2} u}{\partial \varrho^{2}}\right)^{2} \varrho^{n+1} \dr\\&=\frac{1}{2} \int_{\Omega} \int_{0}^{R} V(\varrho)\left(\frac{\partial^{2} u}{\partial \varrho^{2}}\right)^{2}(\sin \phi)^{\frac{n}{2}} \varrho^{n+1} \dr\dph\dw\\&= \int_{\Omega} \int_{0}^{R} V(\varrho)\psi \left(\frac{\partial^{2} u}{\partial \varrho^{2}}\right)^{2}\frac{1}{2}\varrho^{n+1}(\sin \phi)^{\frac{n-2}{2}} \dr\dph\dw\\&=\int_{B_R(o)} V(\varrho)\left|\nabla_{G} u_{\varrho}\right|^{2} \dx\dt\\&=\int_{B_R(o)}W(\varrho)u_\varrho^2|\gradg \varrho|^2\dx\dt+\int_{B_R(o)}V(\varrho)\bigg|\gradg\bigg(\frac{u_\varrho}{f(\varrho)}\bigg)\bigg|^2 f(\varrho)^2\dx\dt\\&=\int_{B_R(o)}W(\varrho)|\gradg u|^2\dx\dt+\int_{B_R(o)}V(\varrho)\bigg|\gradg\bigg(\frac{u_\varrho}{f(\varrho)}\bigg)\bigg|^2 f(\varrho)^2\dx\dt.
	\end{align*}
	Next, consider
	\begin{align*}
		I_{2}&:=\frac{1}{2}|\Omega| \int_{0}^{R}V(\varrho)\frac{(Q-1)^{2}}{\varrho^{2}}\left(\frac{\partial u}{\partial \varrho}\right)^{2}\varrho^{n+1}\dr\\&=\frac{1}{2}\int_{\Omega} \int_{0}^{R}V(\varrho)\frac{(Q-1)^{2}}{\varrho^{2}}\left(\frac{\partial u}{\partial \varrho}\right)^{2}(\sin \phi)^{\frac{n}{2}} \varrho^{n+1} \dr\dph\dw\\&=(Q-1)^2 \int_{\Omega} \int_{0}^{R} \frac{V(\varrho)}{\varrho^{2}} \psi \left(\frac{\partial u}{\partial \varrho}\right)^{2} \frac{1}{2}\varrho^{n+1}(\sin \phi)^{\frac{n-2}{2}} \dr\dph\dw \\
		& =(Q-1)^{2} \int_{B_R(o)} \frac{V(\varrho)}{\varrho^{2}}|\gradg u|^{2} \dx\dt.
	\end{align*}
	Then, using integration by parts, we have
	\begin{align*}
		I_{3}&:=\frac{1}{2}|\Omega| \int_{0}^R 2 V(\varrho) \frac{(Q-1)}{\varrho} \frac{\partial^{2} u}{\partial \varrho^{2}} \frac{\partial u}{\partial \varrho} \varrho^{n+1} \dr \\& = -\frac{1}{2}|\Omega|(Q-1) \int_{0}^R  \frac{V_\varrho(\varrho)}{\varrho} \bigg( \frac{\partial u}{\partial \varrho}\bigg)^2 \varrho^{n+1} \dr \\&-\frac{1}{2}|\Omega|(Q-1) (Q-2)\int_{0}^R  \frac{V(\varrho)}{\varrho^2} \bigg( \frac{\partial u}{\partial \varrho}\bigg)^2 \varrho^{n+1} \dr\\&=-(Q-1) \int_{B_R(o)} \frac{V_\varrho(\varrho)}{\varrho}|\gradg u|^{2} \dx\dt-(Q-1)(Q-2) \int_{B_R(o)} \frac{V(\varrho)}{\varrho^{2}}|\gradg u|^{2} \dx\dt.
	\end{align*}
	Hence, we have
	\begin{align*}
		I_2+I_3=-(Q-1) \int_{B_R(o)} \frac{V_\varrho(\varrho)}{\varrho}|\gradg u|^{2} \dx\dt+(Q-1) \int_{B_R(o)} \frac{V(\varrho)}{\varrho^{2}}|\gradg u|^{2} \dx\dt.
	\end{align*}
	Finally, summing all these we deduce the desired result. $\qed$
	
	\medskip
	
	Next, we shall move on to prove Theorem~\ref{nr-rell}. The proof is quite delicate, and we shall make use of Theorem~\ref{r-rell} for the radial part of the function. Indeed, we exploit the spherical harmonics and the representations of functions on each $\mathcal{H}_k$ to deduce the theorem. We shall begin the proof of Theorem~\ref{nr-rell} below.

	{\bf Proof of Theorem~\ref{nr-rell}:} Again assume $\psi=\sin\phi = |\gradg \varrho|^2$. By spherical harmonics expansion, we can write  	$u(\varrho,\sigma)=\sm d_k(\varrho)\Phi_k(\sigma)$. Using the orthonormality of $\{\Phi_k\}$ on $L^2(\Omega, \dsn)$ we have  
	\begin{align*}
		\int_{\Omega}(\lapg u)^2\psi^{-2}\dsn & =  \sum_{k = 0}^{\infty} \left( d_{k}^{\prime \prime} + \frac{(Q-1)}{\varrho}d_{k}^{\prime} \right)^2  + 
		16\sum_{k = 0}^{\infty} \frac{d_{k}^2}{\varrho^4} \int_{\Omega}(\mathcal{L}_\sigma \Phi_{k})^2\dsn \\
		& + 8 \sum_{k = 0}^{\infty} \left( d_{k}^{\prime \prime} + \frac{(Q-1)}{\varrho}d_{k}^{\prime} \right) \frac{d_{k}}{\varrho^2} \int_{\Omega}(\mathcal{L}_\sigma \Phi_{k}) \Phi_{k}\dsn.
	\end{align*}
	Also, recall that $\{ \Phi_k \}$ is an orthonormal system of spherical harmonics in $L^2(\Omega)$ and satisfies $$-\mathcal{L}_{\sigma}\Phi_k=\lambda_k\Phi_k,$$
	for all $k\in\mathbb{N}\cup\{0\}$, where $\lambda_k=\frac{k(k+n)}{4}$. Now we will divide the proof into a few steps.
	
	{\bf Step 1.} In this step we decompose the l.h.s. of \eqref{nr-rell-eqn} and then by polar coordinate we have
	\begin{align*}
		& \int_{B_{R}(o)} V(\varrho) \frac{(\lapg u)^{2}}{|\gradg \varrho|^{2}} \dx\dt =\frac{1}{2}  \sum_{k = 0}^{\infty} \bigg[\int_0^R V(\varrho) \left( d_{k}^{\prime \prime} + \frac{(Q-1)}{\varrho}d_{k}^{\prime} \right)^2 \varrho^{n+1}\dr \\& + 
		16 \lambda_{k}^2\int_0^R V(\varrho)d_{k}^2 \varrho^{n-3} \dr - 8\lambda_k \int_0^R V(\varrho) \left( d_{k}^{\prime \prime} + \frac{(Q-1)}{\varrho}d_{k}^{\prime} \right) d_k \varrho^{n-1}\dr\bigg].
	\end{align*}
	On the other hand, for each radial function $d_k$, using Theorem \ref{r-rell}, we deduce
	\begin{align*}
		&\int_{B_{R}(o)} V(\varrho) \frac{(\lapg d_k)^{2}}{|\gradg \varrho|^{2}} \dx\dt \\& =\frac{1}{2}|\Omega|\int_0^R V(\varrho) \left( d_{k}^{\prime \prime} + \frac{(Q-1)}{\varrho}d_{k}^{\prime} \right)^2 \varrho^{n+1}\dr\\&=\int_{B_{R}(o)} W(\varrho)|\gradg d_k|^{2}\dx\dt  +(Q-1) \int_{B_{R}(o)}\left(\frac{V(\varrho)}{\varrho^{2}}-\frac{V_{\varrho}(\varrho)}{\varrho}\right)\left|\gradg d_k \right|^2 \dx\dt \\
		& +\int_{B_{R}(o)} V(\varrho)\left|\nabla_{G}\left(\frac{d_k^\prime}{f(\varrho)}\right)\right|^{2} f^{2}(\varrho)\dx\dt\\&=\frac{1}{2}|\Omega|\bigg[\int_0^R W(\varrho){d_k^\prime}^2\varrho^{n+1}\dr+(Q-1)\int_0^R V(\varrho){d_k^\prime}^2\varrho^{n-1}\dr\\&-(Q-1)\int_0^RV_\varrho(\varrho){d_k^\prime}^2\varrho^{n}\dr+\int_0^RV(\varrho)\left|\left(\frac{d_k^\prime}{f(\varrho)}\right)^\prime\right|^{2} f^{2}(\varrho)\varrho^{n+1}\dr\bigg].
	\end{align*}
	Using this in the last identity and after simplifying, we have
	\begin{align*}
		& \int_{B_{R}(o)} V(\varrho) \frac{(\lapg u)^{2}}{|\gradg \varrho|^{2}} \dx\dt =\frac{1}{2}\sum_{k = 0}^{\infty} \bigg[\int_0^R W(\varrho){d_k^\prime}^2\varrho^{n+1}\dr\\&+(Q-1)\int_0^R V(\varrho){d_k^\prime}^2\varrho^{n-1}\dr-(Q-1)\int_0^RV_\varrho(\varrho){d_k^\prime}^2\varrho^{n}\dr\\&+\int_0^RV(\varrho)\left|\left(\frac{d_k^\prime}{f(\varrho)}\right)^\prime\right|^{2} f^{2}(\varrho)\varrho^{n+1}\dr + 
		16 \lambda_{k}^2\int_0^R V(\varrho)d_{k}^2 \varrho^{n-3} \dr \\&- 8\lambda_k \int_0^R V(\varrho)  d_{k}^{\prime \prime} d_k \varrho^{n-1}\dr- 8\lambda_k (Q-1)\int_0^R V(\varrho) d_{k}^{\prime}  d_k \varrho^{n-2}\dr\bigg].
	\end{align*}
	
	{\bf Step 2.} In this step we compute r.h.s. of \eqref{nr-rell-eqn}:
	\begin{align*}
		\int_{B_{R}(o)} &W(\varrho)|\gradg u|^{2}\dx\dt+\int_{B_{R}(o)} V(\varrho)\left|\nabla_{G}\left(\frac{u_\varrho}{f(\varrho)}\right)\right|^{2} f^{2}(\varrho)\dx\dt\\
		& +(Q-1) \int_{B_{R}(o)}\left(\frac{V(\varrho)}{\varrho^{2}}-\frac{V_{\varrho}(\varrho)}{\varrho}\right)\left|\gradg u\right|^2 \dx\dt\\
		&= \frac{1}{2}\sm\bigg[\int_{0}^RW(\varrho){d_k^\prime}^2\varrho^{n+1}\dr+4\lambda_{k}\int_{0}^R W(\varrho)d_k^2\varrho^{n-1}\dr\\
		&+\int_0^RV(\varrho)\left|\left(\frac{d_k^\prime}{f(\varrho)}\right)^\prime\right|^{2} f^{2}(\varrho)\varrho^{n+1}\dr+4\lambda_{k}\int_{0}^R V(\varrho){d_k^\prime}^2\varrho^{n-1}\dr\\
		&+(Q-1)\int_{0}^RV(\varrho){d_k^\prime}^2\varrho^{n-1}\dr+4\lambda_{k}(Q-1)\int_{0}^R V(\varrho)d_k^2\varrho^{n-3}\dr\\
		&-(Q-1)\int_{0}^RV_\varrho(\varrho){d_k^\prime}^2\varrho^{n}\dr-4\lambda_{k}(Q-1)\int_{0}^R V_\varrho(\varrho)d_k^2\varrho^{n-2}\dr\bigg].
	\end{align*}
	
	{\bf Step 3.} Subtracting the r.h.s. of the identities obtained in Step 1 and Step 2, we obtain the expression below that we denote by $\mathcal{R}$ the following quantity:
	\begin{align*}
		\mathcal{R}:=&\sm\bigg[4\lambda_k (4\lambda_{k}-Q+1)\int_0^R V(\varrho)d_{k}^2 \varrho^{n-3} \dr - 8\lambda_k \int_0^R V(\varrho)  d_{k}^{\prime \prime} d_k \varrho^{n-1}\dr\\
		&- 8\lambda_k (Q-1)\int_0^R V(\varrho) d_{k}^{\prime}  d_k \varrho^{n-2}\dr
		-4\lambda_{k}\int_{0}^R W(\varrho)d_k^2\varrho^{n-1}\dr\\
		&-4\lambda_{k}\int_{0}^R V(\varrho){d_k^\prime}^2\varrho^{n-1}\dr
		+4\lambda_{k}(Q-1)\int_{0}^R V_\varrho(\varrho)d_k^2\varrho^{n-2}\dr\bigg].
	\end{align*}
	In the steps below, we will show $\mathcal{R}$ is non-negative and this will establish the result. To do this, we need some identities.
	
	{\bf Step 4.} Set
	\begin{align*}
		\mathcal{I}:=\int_{0}^R V(\varrho){d_k^\prime}^2\varrho^{n-1}\dr,
	\end{align*}
	\begin{align*}
		\mathcal{I}_1:=\int_0^R V(\varrho) d_{k}^{\prime}  d_k \varrho^{n-2}\dr,
	\end{align*}
	and 
	\begin{align*}
		\mathcal{I}_2:= \int_0^R V(\varrho)  d_{k}^{\prime \prime} d_k \varrho^{n-1}\dr.
	\end{align*}
	Now, using the integration by parts, we have
	\begin{align*}
		\mathcal{I}_1=-\frac{1}{2}\int_0^R V_\varrho(\varrho) d_{k}^{2}  \varrho^{n-2}\dr-\frac{(Q-4)}{2}\int_0^R V(\varrho) d_{k}^{2}  \varrho^{n-3}\dr,
	\end{align*}
	and
	\begin{align*}
		\mathcal{I}_2&=-\int_0^R V_\varrho(\varrho)  d_{k}^{\prime} d_k \varrho^{n-1}\dr-\mathcal{I}-(Q-3)\mathcal{I}_1\\&=\frac{1}{2}\int_0^R V_{\varrho\varrho}(\varrho)  d_{k}^2 \varrho^{n-1}\dr+\frac{(Q-3)}{2}\int_0^R V_{\varrho}(\varrho)  d_{k}^2 \varrho^{n-2}\dr-\mathcal{I}-(Q-3)\mathcal{I}_1.
	\end{align*}
	
	{\bf Step 5.} In this step we will deal with $\mathcal{I}$. Let $b_k(\varrho):=\frac{d_k(\varrho)}{\varrho}$, by Leibniz rule, performing the derivative w.r.t $\varrho$, we have $d_k^{\prime}=\varrho b_{k}^{\prime}+b_k$. Using this and the integration by parts formula, we obtain
	\begin{align}\label{nrad_rellich_2}
		\mathcal{I}&=\int_{0}^R V(\varrho){d_k^\prime}^2\varrho^{n-1}\dr\notag\\&=\int_{0}^{R}V(\varrho) {b_{k}^{\prime}}^2\varrho^{n+1}\dr-\int_{0}^{R}V_\varrho(\varrho)b_k^2\varrho^{n}\dr-(Q-3)\int_{0}^{R}V(\varrho)b_k^2\varrho^{n-1}\dr.
	\end{align}
	Applying Theorem \ref{hardy-rell-iden}, for the radial function $b_k$, we deduce
	\begin{align*}
		\int_{0}^{R}V(\varrho) {b_{k}^{\prime}}^2\varrho^{n+1}\dr&= \int_{0}^{R}W(\varrho)b_k^2\varrho^{n+1}\dr+\int_{0}^R V(\varrho)f^2(\varrho) \left[\bigg(\frac{b_k}{f(\varrho)}\bigg)'\right]^2 \varrho^{n+1}\dr.
	\end{align*}
	Using this estimate into \eqref{nrad_rellich_2} and writing $b_k$ in terms of $d_k$, we have
	\begin{align*}
		\mathcal{I}&= \int_{0}^{R}W(\varrho)b_k^2\varrho^{n+1}\dr+\int_{0}^R V(\varrho)f^2(\varrho) \left[\bigg(\frac{b_k}{f(\varrho)}\bigg)'\right]^2 \varrho^{n+1}\dr\\&-\int_{0}^{R}V_\varrho(\varrho)b_k^2\varrho^{n}\dr-(Q-3)\int_{0}^{R}V(\varrho)b_k^2\varrho^{n-1}\dr\\&=\int_{0}^{R}W(\varrho)d_k^2\varrho^{n-1}\dr+\int_{0}^R V(\varrho)f^2(\varrho) \left[\bigg(\frac{d_k}{\varrho f(\varrho)}\bigg)'\right]^2 \varrho^{n+1}\dr\\&-\int_{0}^{R}V_\varrho(\varrho)d_k^2\varrho^{n-2}\dr-(Q-3)\int_{0}^{R}V(\varrho)d_k^2\varrho^{n-3}\dr.
	\end{align*}
	
	{\bf Step 6.} Next, using all these integrals and rewriting $\mathcal{R}$ as follows:
	\begin{align*}
		\mathcal{R}=&\sm\bigg[4\lambda_k (4\lambda_{k}-Q+1)\int_0^R V(\varrho)d_{k}^2 \varrho^{n-3} \dr - 8\lambda_k \mathcal{I}_2- 8\lambda_k (Q-1)\mathcal{I}_1\\
		&-4\lambda_{k}\int_{0}^R W(\varrho)d_k^2\varrho^{n-1}\dr-4\lambda_{k}\mathcal{I}
		+4\lambda_{k}(Q-1)\int_{0}^R V_\varrho(\varrho)d_k^2\varrho^{n-2}\dr\bigg]\\
		&=\sm\bigg[4\lambda_k (4\lambda_{k}-Q+1)\int_0^R V(\varrho)d_{k}^2 \varrho^{n-3} \dr - 8\lambda_k \biggl\{\frac{1}{2}\int_0^R V_{\varrho\varrho}(\varrho)  d_{k}^2 \varrho^{n-1}\dr\\
		&+\frac{(Q-3)}{2}\int_0^R V_{\varrho}(\varrho)  d_{k}^2 \varrho^{n-2}\dr-\mathcal{I}-(Q-3)\mathcal{I}_1\biggr\}- 8\lambda_k (Q-1)\mathcal{I}_1\\
		&-4\lambda_{k}\int_{0}^R W(\varrho)d_k^2\varrho^{n-1}\dr-4\lambda_{k}\mathcal{I}
		+4\lambda_{k}(Q-1)\int_{0}^R V_\varrho(\varrho)d_k^2\varrho^{n-2}\dr\bigg]\\
		&=\sm\bigg[4\lambda_k (4\lambda_{k}-Q+1)\int_0^R V(\varrho)d_{k}^2 \varrho^{n-3} \dr - 4\lambda_k\int_0^R V_{\varrho\varrho}(\varrho)  d_{k}^2 \varrho^{n-1}\dr\\
		&+8\lambda_{k}\int_0^R V_{\varrho}(\varrho)  d_{k}^2 \varrho^{n-2}\dr+4\lambda_{k}\mathcal{I}-16\lambda_k\mathcal{I}_1-4\lambda_{k}\int_{0}^R W(\varrho)d_k^2\varrho^{n-1}\dr\\
		&=\sm\bigg[4\lambda_k (4\lambda_{k}+Q-7)\int_0^R V(\varrho)d_{k}^2 \varrho^{n-3} \dr - 4\lambda_k\int_0^R V_{\varrho\varrho}(\varrho)  d_{k}^2 \varrho^{n-1}\dr\\
		&+16\lambda_{k}\int_0^R V_{\varrho}(\varrho)  d_{k}^2 \varrho^{n-2}\dr-4\lambda_{k}\int_{0}^R W(\varrho)d_k^2\varrho^{n-1}\dr+4\lambda_{k}\mathcal{I}
		\bigg].
	\end{align*}
	Now we will substitute the value of $\mathcal{I}$, and we obtain after further simplifying
	\begin{align*}
		\mathcal{R}&=\sm\bigg[16\lambda_k (\lambda_{k}-1)\int_0^R V(\varrho)d_{k}^2 \varrho^{n-3} \dr - 4\lambda_k\int_0^R V_{\varrho\varrho}(\varrho)  d_{k}^2 \varrho^{n-1}\dr\\
		&+12\lambda_{k}\int_0^R V_{\varrho}(\varrho)  d_{k}^2 \varrho^{n-2}\dr+4\lambda_{k}\int_{0}^R V(\varrho)f^2(\varrho) \left[\bigg(\frac{d_k}{\varrho f(\varrho)}\bigg)'\right]^2 \varrho^{n+1}\dr\bigg]\\
		&\geq \sm 4\lambda_{k}\int_0^R \bigg[4(\lambda_{k}-1) \frac{V(\varrho)}{\varrho^2}  - V_{\varrho\varrho}(\varrho) +3 \frac{V_{\varrho}(\varrho) }{\varrho} \bigg]d_{k}^2\varrho^{n-1} \dr.
	\end{align*}
	In the above, we remove the non-negative term due to $V\geq 0$. Also, we can notice that for $k\geq 1$ there holds
	\begin{align*}
		4(\lambda_{k}-1)=k(n+k)-4\geq (n+1)-4 = (Q-5).
	\end{align*}
	With the given condition \eqref{cond} and the above estimate we conclude the proof. $\qed$
	
	\medspace
	
	{\bf Proof of Corollary~\ref{cor-hr-r} and Corollary~\ref{cor-hr-r-rad}:} Applying $\big(1,\frac{(Q-2)^2}{4}\frac{1}{\varrho^2}\big)$ on Theorem \ref{nr-rell} as a $Q$-dimensional Bessel pair on $(0,\infty)$ with positive solution $\varrho^{-\frac{Q-2}{2}}$, part (a) follows. Now applying weighted Hardy inequality from Corollary \ref{wg-hardy}, with $\alpha=2$, we deduce
	\begin{align*}
		\int_{\rno} \frac{|\gradg u|^2}{\varrho^2}\dx\dt&-\frac{(Q-4)^2}{4}\int_{\rno}\frac{|u|^2|\gradg \varrho|^2}{\varrho^{4}}\dx\dt\\&=\int_{\rno}\varrho^{2-Q}\left|\gradg\left(u\varrho^{\frac{Q-4}{2}}\right)\right|^2\dx\dt.
	\end{align*}
	Clubbing these with part (a) we obtain our required Rellich inequality in part (b).
	
	\medskip
	
	To prove the Corollary~\ref{cor-hr-r-rad}, we will follow the same approach as Corollary \ref{cor-hr-r}, and utilize Theorem \ref{r-rell}. This completes the proof. $\qed$
	
	\medskip			
	
	\begin{remark}
		{\rm
			Using the polar coordinate decomposition and integration by parts formula, we can have the following identity: for the radial function $u\in C_c^\infty(\rno \setminus \{o\})$, there holds
			\begin{align*}
				\frac{Q^2}{4}\int_{\rno}\varrho^{2-Q}&\left|\gradg\left(u\varrho^{\frac{Q-4}{2}}\right)\right|^2\dx\dt
				+\int_{\rno} \varrho^{2-Q}\left|\nabla_{G}\left(u_\varrho \varrho^{\frac{Q-2}{2}}\right)\right|^{2} \dx\dt\\&= \int_{\rno} \bigg|\frac{\lapg u}{|\gradg \varrho|}+\frac{Q(Q-4)}{4}\frac{u |\gradg \varrho|}{\varrho^2}\bigg|^2\dx\dt \\& +\frac{Q(Q-4)}{2}\int_{\rno}\bigg|\frac{|\nabla_{G} u|}{\varrho|\gradg \varrho|}+\frac{(Q-4)}{2}\frac{u }{\varrho^2}\bigg|^2|\gradg \varrho|^2\dx\dt.	
			\end{align*}
			We can compare the corresponding Rellich identity with the radial Grushin operator inspired by \cite[Theorem 1.1]{mow}. For the sake of brevity, we omit the details. 
		}
	\end{remark}
	
	\medspace
	
	\section{Rellich identity and spherical harmonics via spherical derivatives: proof of Theorem~\ref{rellichspherical} }\label{sec-sph-deri}	
	This section is devoted to studying spherical harmonics via spherical derivatives. We shall explicitly compute the contribution of the radial and the spherical part of the Grushin operator.  Let us recall the spherical derivatives, for $j =1, \ldots, n$ we define $
	L_j=\frac{\partial}{\partial x_j}-\frac{\partial\varrho}{\partial x_j}\frac{\partial}{\partial \varrho}$ \mbox{and} 
	$L_{n+1}=|x|\frac{\partial}{\partial t}-|x|\frac{\partial\varrho}{\partial t}\frac{\partial}{\partial \varrho}.
	$
	Using polar coordinates, it is easy to check that \begin{equation}\label{0}
		\frac{\partial}{\partial \varrho}=\frac{1}{\varrho}\left(\sum_{j=1}^n x_j\frac{\partial}{\partial x_j}+2t\frac{\partial}{\partial t}\right).
	\end{equation}
	
	Now, we will state some useful identities involving $L_j$'s in the following lemma.
	\begin{lemma}\label{iden-lem}
		The following relations hold:
		\begin{itemize}
			\item[1.] There holds
			\begin{equation}\label{2}
				\sum_{j=1}^n x_j|x|^2L_j+2t|x|L_{n+1}=0.
			\end{equation}
			\item[2.] For all $j=1,2,\ldots,n+1$, we have
			\begin{equation}\label{3}
				L_j\frac{\partial}{\partial \varrho}=\frac{\partial}{\partial \varrho}L_j+\frac{1}{\varrho}L_j.
			\end{equation}
			\item[3.] We have the following identities
			\begin{equation}\label{1}
				\lapg = \rlapg +\sum_{j=1}^{n+1}L_j^2, \quad \text{ and } \quad \sum_{j=1}^{n+1}L_j^2=\frac{4\psi}{\varrho^2} \mathcal{L}_\sigma. 
			\end{equation}
			\item[4.] Let $f, \, g \in C_c^2(\rno)$. Then the following integration by parts formula holds: when $j=1,\ldots,n$, we have
			$$\int_{\rno}g (L_jf)\dx\dt=-\int_{\rno}(L_jg){f}\dx\dt+(Q-1)\int_{\rno}g\frac{x_j|x|^2}{\varrho^4}{f} \dx\dt,$$
			and when $j=n+1$, we get 
			$$\int_{\rno}g(L_{n+1}f)\dx\dt=-\int_{\rno}(L_{n+1}g){f}\dx\dt+(Q-1)\int_{\rno}g\frac{2t|x|}{\varrho^4}{f} \dx\dt.$$
			\item[5.] Assume $f\in C^2(\rno)$. Then for all $j$ and $\alpha\in\mathbb{R}$, there holds $$L_j(\varrho^{\alpha}f)=\varrho^{\alpha}(L_jf).$$
		\end{itemize}
	\end{lemma}
	\begin{proof}
		Let us begin the proof of part (1). Using the definition of $L_j$'s and \eqref{0}, we have
		\begin{align*}
			\sum_{j=1}^n x_j|x|^2L_j+2t|x|L_{n+1}&=\sum_{j=1}^n x_j|x|^2\frac{\partial}{\partial x_j}-\sum_{j=1}^n x_j|x|^2\frac{\partial\varrho}{\partial x_j}\frac{\partial}{\partial \varrho}+2t|x|^2\frac{\partial}{\partial t}-2t|x|^2\frac{\partial\varrho}{\partial t}\frac{\partial}{\partial \varrho}\\&=\varrho|x|^2\frac{\partial}{\partial \varrho}-|x|^2 \left(\sum_{j=1}^n \frac{x_j^2|x|^2}{\varrho^3}+\frac{4t^2}{\varrho^3}\right)\frac{\partial}{\partial \varrho}=0.
		\end{align*}
		
		\medspace
		
		Now we will prove part (2). Using \eqref{0} and the fact that $\frac{\partial\varrho}{\partial x_j};j=1,2, \ldots, n$, is independent of $\varrho$, we can notice that
		\begin{align*}
			L_j\frac{\partial}{\partial \varrho}-\frac{\partial}{\partial \varrho}L_j&=\left(\frac{\partial }{\partial x_j}\right)\left(\sum_{k=1}^n \frac{x_k}{\varrho}\frac{\partial}{\partial x_k}+\frac{2t}{\varrho}\frac{\partial}{\partial t}\right)-\left(\sum_{k=1}^n \frac{x_k}{\varrho}\frac{\partial}{\partial x_k}+\frac{2t}{\varrho}\frac{\partial}{\partial t}\right)\left(\frac{\partial }{\partial x_j}\right)\\&=\sum_{k=1}^n\frac{x_k}{\varrho}\frac{\partial^2}{\partial x_j\partial x_k}+\frac{2t}{\varrho}\frac{\partial}{\partial x_j}\frac{\partial}{\partial t}-\sum_{k=1}^n\frac{x_k}{\varrho}\frac{\partial^2}{\partial x_k\partial x_j}-\frac{2t}{\varrho}\frac{\partial}{\partial t}\frac{\partial}{\partial x_j}\\&+\frac{1}{\varrho}\frac{\partial }{\partial x_j}-\sum_{k=1}^n\frac{x_kx_j|x|^2}{\varrho^5}\frac{\partial}{\partial x_k}-\frac{2tx_j|x|^2}{\varrho^5}\frac{\partial }{\partial t}\\&=\frac{1}{\varrho}\bigg[ \frac{\partial }{\partial x_j}-\frac{\partial \varrho}{\partial x_j}\bigg(\sum_{k=1}^n\frac{x_k}{\varrho}\frac{\partial}{\partial x_k}+\frac{2t}{\varrho}\frac{\partial }{\partial t}\bigg) \bigg].
		\end{align*}
		Now, using \eqref{0}, and writing back the definition of $L_j$, we obtain the result. The argument for $j=n+1$ is left to the reader.
		
		\medspace
		
		Now we will prove part (3).  Let us calculate $L_j^2$ first. For $j=1,2,\ldots,n$, we have 
		\begin{align*}
			L_j^2=\frac{\partial^2}{\partial x_j^2}	-\frac{\partial\varrho}{\partial x_j}\frac{\partial}{\partial\varrho}\frac{\partial}{\partial x_j}-\frac{\partial}{\partial x_j}\left(\frac{\partial\varrho}{\partial x_j}\frac{\partial}{\partial\varrho}\right)+ \frac{\partial\varrho}{\partial x_j}\frac{\partial}{\partial\varrho}\left(\frac{\partial\varrho}{\partial x_j}\frac{\partial}{\partial\varrho}\right).
		\end{align*}
		Using the fact that $\frac{\partial\varrho}{\partial x_j};j=1,2, \ldots, n$, and $|x|\frac{\partial\varrho}{\partial t}$ is independent of $\varrho$ and further simplifying the above expression of $L_j^2$ we get
		\begin{align*}
			L_j^2=\frac{\partial^2}{\partial x_j^2}+\left(\frac{\partial\varrho}{\partial x_j}\right)^2\frac{\partial^2}{\partial\varrho^2}-\frac{|x|^2}{\varrho^3}-\frac{2x_j^2}{\varrho^3}+\frac{3x_j^2|x|^4}{\varrho^7}	\frac{\partial}{\partial\varrho}-\frac{\partial\varrho}{\partial x_j}\frac{\partial}{\partial x_j}\frac{\partial}{\partial \varrho}-\frac{\partial\varrho}{\partial x_j}\frac{\partial}{\partial \varrho}\frac{\partial}{\partial x_j}.	
		\end{align*}
		Similarly, we can compute that
		\begin{align*}
			L_{n+1}^2=|x|^2\frac{\partial^2}{\partial t^2}+|x|^2\left(\frac{\partial\varrho}{\partial t}\right)^2\frac{\partial^2}{\partial\varrho^2}	-\frac{2|x|^2}{\varrho^3}\frac{\partial}{\partial\varrho}+\frac{12|x|^2t}{\varrho^7}\frac{\partial}{\partial\varrho}\\-|x|^2\frac{\partial\varrho}{\partial t}\frac{\partial}{\partial t}\frac{\partial}{\partial \varrho}-|x|^2\frac{\partial\varrho}{\partial t}\frac{\partial}{\partial\varrho}\frac{\partial}{\partial t}-\frac{|x|^2}{\varrho}\frac{\partial\varrho}{\partial t}\frac{\partial}{\partial t}.
		\end{align*}
		Then, taking the summation of the above identities from $j=1,2,\ldots, n+1$, we get
		\begin{equation}\label{sphex}
			\sum_{j=1}^{n+1}L_j^2=	\lapg +|\gradg\varrho|^2\frac{\partial^2}{\partial\varrho^2}-\frac{n+1}{\varrho}|\gradg\varrho|^2\frac{\partial}{\partial\varrho}-D,
		\end{equation}
		where
		\begin{align*}
			D=\sum_{j=1}^n\frac{\partial\varrho}{\partial x_j}\frac{\partial}{\partial x_j}\frac{\partial}{\partial \varrho}+|x|^2\frac{\partial\varrho}{\partial t}\frac{\partial}{\partial t}\frac{\partial}{\partial \varrho}+\sum_{j=1}^n\frac{\partial\varrho}{\partial x_j}\frac{\partial}{\partial \varrho}\frac{\partial}{\partial x_j}+|x|^2\frac{\partial\varrho}{\partial t}\frac{\partial}{\partial\varrho}\frac{\partial}{\partial t}+\frac{|x|^2}{\varrho}\frac{\partial\varrho}{\partial t}\frac{\partial}{\partial t}	.
		\end{align*}
		Using the identity \eqref{0}, and the fact that $\frac{x_j}{\varrho}$, and $\frac{|x|^2}{\varrho^2}$ are $\varrho$ independent, we obtain
		\begin{align*}
			\sum_{j=1}^n\frac{\partial\varrho}{\partial x_j}\frac{\partial}{\partial x_j}\frac{\partial}{\partial \varrho}+|x|^2\frac{\partial\varrho}{\partial t}\frac{\partial}{\partial t}\frac{\partial}{\partial \varrho}=|\gradg \varrho|^2\frac{\partial^2}{\partial \varrho^2},
		\end{align*} 
		and
		\begin{align*}
			\sum_{j=1}^n\frac{\partial\varrho}{\partial x_j}\frac{\partial}{\partial \varrho}\frac{\partial}{\partial x_j}=\frac{|x|^2}{\varrho^2}\frac{\partial}{\partial \varrho}\left(\frac{\partial}{\partial \varrho}-\frac{2t}{\varrho}\frac{\partial }{\partial t}\right)=|\gradg \varrho|^2\frac{\partial^2}{\partial \varrho^2}-\frac{\partial}{\partial\varrho}\left(|x|^2\frac{\partial\varrho}{\partial t}\frac{\partial}{\partial t}\right).
		\end{align*}
		Again noticing that $|x|\frac{\partial\varrho}{\partial t}$ is independent of $\varrho$, and using \eqref{0} to compute $\frac{\partial }{\partial \varrho}(|x|)$, we obtain
		\begin{align*}
			\frac{\partial}{\partial\varrho}\left(|x|^2\frac{\partial\varrho}{\partial t}\frac{\partial}{\partial t}\right)=\frac{|x|^2}{\varrho}\frac{\partial\varrho}{\partial t}\frac{\partial}{\partial t}	+|x|^2\frac{\partial\varrho}{\partial t}\frac{\partial}{\partial\varrho}\frac{\partial}{\partial t}.
		\end{align*}
		Combining all these in the expression of $D$, it simplifies to $D=2|\gradg \varrho|^2\frac{\partial^2}{\partial\varrho^2}$. Putting this expression of $D$ in \eqref{sphex} we get that $$\lapg = \rlapg +\sum_{j=1}^{n+1}L_j^2.$$ Using the expression of the Grushin operator in \eqref{grushin-polar}, we can obtain the second identity of \eqref{1} and this completes the proof of part (3).
		
		\medspace
		
		Now we will prove part (4) of the lemma. This follows from the definition of $L_j$'s and the polar coordinate decomposition formula. Exactly similar arguments hold for $j=n+1$, and for simplicity, we are just writing the case when $j=1,2,\ldots,n$. Exploiting integration by parts, first w.r.t. $x_j$, and then w.r.t. $\varrho$, and using the fact that $\frac{\partial\varrho}{\partial x_j}$ is independent of $\varrho$, we have
		\begin{align*}
			\int_{\rno}g (L_jf)\dx\dt&=-\int_{\rno}\frac{\partial g}{\partial x_j} \, f\dx\dt -\int_{\Omega}\int_{0}^{\infty}g\frac{\partial\varrho}{\partial x_j}\frac{\partial f}{\partial \varrho}\frac{\varrho^{Q-1}}{2\psi}\dr\dsn\\&=-\int_{\rno}(L_jg){f}\dx\dt+(Q-1)\int_{\rno}\frac{g}{\varrho}\frac{\partial\varrho}{\partial x_j}f \dx\dt.
		\end{align*}
		
		\medspace
		
		In this part, we will prove part (5). Notice that for any $j=1,2,\ldots,n$, we have
		\begin{align*}
			L_j (\varrho^\alpha f )=\frac{\partial \varrho^\alpha}{\partial x_j} f+\varrho^\alpha\frac{\partial f}{\partial x_j}-\frac{\partial\varrho}{\partial x_j}\frac{\partial \varrho^\alpha}{\partial \varrho} f - \frac{\partial\varrho}{\partial x_j} \varrho^\alpha \frac{\partial f}{\partial \varrho}=\varrho^\alpha (L_j f).
		\end{align*}
		A similar argument holds for $j=n+1$. This finishes the proof of the lemma.
	\end{proof}
	
	Now we begin the proof of  Theorem~\ref{rellichspherical}.
	
	{\bf Proof of Theorem~\ref{rellichspherical}:} Using \eqref{1} we can write 
	\begin{align*}
		\int_{\rno}\frac{|\mathcal{L}_{G} u|^2}{|\gradg \varrho|^2}\dx \dt= \int_{\rno}\frac{|\rlapg u|^2}{|\gradg \varrho|^2}\dx \dt	+\int_{\rno}\frac{|\sum_{j=1}^{n+1}L_j^2 u|^2}{|\gradg \varrho|^2}\dx \dt\\
		+2~\left(\sum_{j=1}^{n+1}\int_{\rno}|\gradg \varrho|^{-2}\rlapg u {L_j^2u}\dx\dt\right).
	\end{align*}
	We know that $$|\gradg \varrho|^{-2}\rlapg u =\frac{\partial^2 u}{\partial\varrho^2}+\frac{Q-1}{\varrho}\frac{\partial u}{\partial\varrho}.$$
	
	Define $L_ju=f_j$ and we write $\frac{\partial^2 u}{\partial\varrho^2}+\frac{Q-1}{\varrho}\frac{\partial u}{\partial\varrho}=g$, then part (4) of Lemma~\ref{iden-lem} gives $$\int_{\rno}g{L_jf_j}\dx\dt=-\int_{\rno}L_jg{f_j}\dx\dt+(Q-1)\int_{\rno}g\frac{x_j|x|^2}{\varrho^4}{f_j} \dx\dt$$ for $j=1,2,\ldots, n$, and for $j=n+1$, we have
	$$\int_{\rno}g{L_{n+1}f_{n+1}}\dx\dt=-\int_{\rno}L_{n+1}g{f_{n+1}}\dx\dt+(Q-1)\int_{\rno}g\frac{2t|x|}{\varrho^4}{f_{n+1}} \dx\dt.$$
	
	Now summing over $j$ and using \eqref{2} gives 
	\begin{equation}\label{4}
		\sum_{j=1}^{n+1}\int_{\rno}|\gradg \varrho|^{-2}\rlapg u {L_ju}\dx\dt=-\sum_{j=1}^{n+1}\int_{\rno}L_jg{f_j}\dx\dt.
	\end{equation}
	
	By the definition of $g$ above it can be rewritten as $g=\varrho^{-(Q-1)}\frac{\partial}{\partial\varrho}\left(\varrho^{Q-1}\frac{\partial u}{\partial\varrho}\right)$. Using \eqref{3} and part (5) of Lemma~\ref{iden-lem}, we can write $$L_jg=\frac{\partial^2 f_j}{\partial \varrho^2}+\frac{Q+1}{\varrho}\frac{\partial f_j}{\partial \varrho}+\frac{Q-1}{\varrho^2}f_j.$$
	
	For a fixed $j$ we have \begin{align*}\int_{\rno}L_jg{f_j}\dx\dt=\int_{\rno}\frac{\partial^2 f_j}{\partial \varrho^2}f_j+\frac{Q+1}{\varrho}\frac{\partial f_j}{\partial \varrho}f_j+\frac{Q-1}{\varrho^2}f_j^2\dx\dt.
	\end{align*}
	
	Using integration by parts gives \begin{align*}\int_{\rno}\frac{\partial^2 f_j}{\partial \varrho^2}f_j+\frac{Q+1}{\varrho}\frac{\partial f_j}{\partial \varrho}f_j\dx\dt=-\int_{\rno}\left(\frac{\partial f_j}{\partial\varrho}\right)^2\dx\dt\\-\frac{(Q+1)(Q-2)}{2}\int_{\rno}\frac{f_j^2}{\varrho^2}\dx\dt+\frac{(Q-1)(Q-2)}{2}\int_{\rno}\frac{f_j^2}{\varrho^2}\dx\dt.\end{align*}
	Summing over $j$ gives
	$$\sum_{j=1}^{n+1}\int_{\rno}L_jg{f_j}\dx\dt=-\sum_{j=1}^{n+1}\left(\int_{\rno}\left|\frac{\partial f_j}{\partial\varrho}\right|^2\dx\dt-\int_{\rno}\frac{|f_j|^2}{\varrho^2}\dx\dt\right).$$
	Therefore, \eqref{4} reduces to		  
	$$\sum_{j=1}^{n+1}\int_{\rno}|\gradg \varrho|^{-2}\rlapg u {L_ju}\dx\dt=\sum_{j=1}^{n+1}\left(\int_{\rno}\left|\frac{\partial f_j}{\partial\varrho}\right|^2\dx\dt-\int_{\rno}\frac{|f_j|^2}{\varrho^2}\dx\dt\right).$$
	
	Using the identity for $\rgradg$ when $\alpha=0$, in Corollary \ref{wg-hardy} for $|\gradg \varrho|^{-1}L_ju$, $j=1\ldots n$, we get the desired equality. $\qed$
	
	\medskip	
	
	Now we want to measure the following deficit
	\begin{align*}
		\bigg|\bigg|\frac{\lapg u}{\psi^{1/2}}\bigg|\bigg|_{L^2(\rno)}-\quad	\bigg|\bigg|\frac{\rlapg u}{\psi^{1/2}}\bigg|\bigg|_{L^2(\rno)},
	\end{align*}
	where $u\in C_c^\infty(\rno\setminus\{o\})$. Here, we will calculate the exact identity in terms of spherical harmonics. This can be compared with the second part of the result \cite[Theorem 1]{bmo}. Our result can be stated as follows.
	\begin{theorem}\label{rellichprojection}
		Let $Q\geq 4$ and $u\in C_c^\infty(\rno\setminus\{o\})$. Then there holds
		\begin{align*}
			&\norm\frac{\lapg u}{\psi^{1/2}}\norm^2_{L^2(\rno)}=\quad\norm\frac{\rlapg u}{\psi^{1/2}}\norm^2_{L^2(\rno)} + 16\sm \lambda_k^2\norm\frac{(\mathcal{P}_k u)\psi^{1/2}}{\varrho^2}\norm^2_{L^2(\rno)}\\&+8\sm \lambda_k\norm\frac{(\mathcal{P}_k u_\varrho)\psi^{1/2}}{\varrho}\norm^2_{L^2(\rno)}+8(Q-4)\sm \lambda_k\norm\frac{(\mathcal{P}_k u)\psi^{1/2}}{\varrho^2}\norm^2_{L^2(\rno)}.
		\end{align*}
		In the above, $\varrho$ in suffix means the derivative with respect to the radial variable $\varrho$.
	\end{theorem}
	\begin{proof}
		Consider the decomposition $u(\varrho,\sigma)=\sum_{k=0}^{\infty}\mathcal{P}_k u$ where $\mathcal{P}_k u=d_{k}(\varrho)\Phi_k(\sigma)$. Then let us write the Grushin operator as follows
		\begin{align*}
			\lapg u&=\psi \biggl\{\frac{\partial^2 u}{\partial \varrho^2}+\frac{Q-1}{\varrho}\frac{\partial u}{\partial \varrho}+\frac{4}{\varrho^2}\mathcal{L}_\sigma u\biggr\}\\&=\rlapg u + \frac{4\psi}{\varrho^2}\mathcal{L}_\sigma\bigg( \sum_{k=0}^{\infty}d_{k}(\varrho)\Phi_k(\sigma)\bigg)\\&=\rlapg u - \frac{4\psi}{\varrho^2} \sum_{k=0}^{\infty}\lambda_k d_{k}(\varrho)\Phi_k(\sigma).
		\end{align*}
		If we now expand it, then we deduce using orthonormality properties that
		\begin{align*}
			\norm\frac{\lapg u}{\psi^{1/2}}\norm^2_{L^2(\rno)}&=\norm\frac{\rlapg u}{\psi^{1/2}}\norm^2_{L^2(\rno)}+\bigg|\bigg|\frac{4\psi^{1/2}}{\varrho^2} \sum_{k=0}^{\infty}\lambda_k d_{k}(\varrho)\Phi_k(\sigma)\bigg|\bigg|^2_{L^2(\rno)}\\&-2\left\langle\frac{\rlapg u}{\psi^{1/2}},\frac{4\psi^{1/2}}{\varrho^2} \sum_{k=0}^{\infty}\lambda_k d_{k}(\varrho)\Phi_k(\sigma)\right\rangle_{L^2(\rno)}.
		\end{align*}
		Now we will compute each term separately. Let us begin with
		\begin{align*}
			&\bigg|\bigg|\frac{4\psi^{1/2}}{\varrho^2} \sum_{k=0}^{\infty}\lambda_k d_{k}(\varrho)\Phi_k(\sigma)\bigg|\bigg|^2_{L^2(\rno)}\\&=\int_{\Omega}\int_{0}^{\infty}\bigg(\frac{4\psi^{1/2}}{\varrho^2} \sum_{k=0}^{\infty}\lambda_k d_{k}(\varrho)\Phi_k(\sigma)\bigg)\bigg(\frac{4\psi^{1/2}}{\varrho^2} \sum_{k=0}^{\infty}\lambda_k d_{k}(\varrho)\Phi_k(\sigma)\bigg)\frac{\varrho^{n+1}}{2\psi}\dr\dsn\\&=\sm 8 \lambda_k^2 \int_{0}^\infty \varrho^{n-3}d_{k}^2\dr,
		\end{align*}
		and
		\begin{align*}
			&\left\langle \frac{\rlapg u}{\psi^{1/2}},\frac{4\psi^{1/2}}{\varrho^2} \sum_{k=0}^{\infty}\lambda_k d_{k}(\varrho)\Phi_k(\sigma)\right\rangle_{L^2(\rno)}\\&=\int_{\Omega}\int_{0}^{\infty}\bigg(\psi^{1/2}\sm\biggl\{\frac{\partial^2 d_{k}}{\partial \varrho^2}+\frac{Q-1}{\varrho}\frac{\partial d_{k}}{\partial \varrho}\biggr\}\Phi_k\bigg)\bigg(\frac{4\psi^{1/2}}{\varrho^2} \sum_{k=0}^{\infty}\lambda_k d_{k}\Phi_k\bigg)\frac{\varrho^{n+1}}{2\psi}\dr\dsn\\&=\sm 2 \int_{0}^\infty \varrho^{n-1} \lambda_k \big[d_{k}^{\prime\prime}+\frac{Q-1}{\varrho}d_{k}^\prime\big] d_{k} \dr\\&=2\sm \lambda_k\int_{0}^\infty\varrho^{n-1}d_{k}^{\prime\prime}d_{k}\dr+2(Q-1)\sm \lambda_k \int_{0}^\infty\varrho^{n-2}d_{k}^{\prime}d_{k}\dr.
		\end{align*}
		Notice that,
		\begin{align*}
			\int_{0}^\infty\varrho^{n-2}d_{k}^{\prime}d_{k}\dr=-\frac{(Q-4)}{2}\int_{0}^\infty\varrho^{n-3}d_{k}^2\dr,
		\end{align*}
		and
		\begin{align*}
			\int_{0}^\infty\varrho^{n-1}d_{k}^{\prime\prime}d_{k}\dr&=-\int_{0}^\infty\varrho^{n-1}(d_{k}^{\prime})^2\dr-(Q-3)\int_{0}^\infty\varrho^{n-2}d_{k}^{\prime}d_{k}\dr\\&=-\int_{0}^\infty\varrho^{n-1}(d_{k}^{\prime})^2\dr+\frac{(Q-3)(Q-4)}{2}\int_{0}^\infty\varrho^{n-3}d_{k}^2\dr.
		\end{align*}
		Combining all these we deduce
		\begin{align*}
			&\left\langle \frac{\rlapg u}{\psi^{1/2}},\frac{4\psi^{1/2}}{\varrho^2} \sum_{k=0}^{\infty}\lambda_k d_{k}(\varrho)\Phi_k(\sigma)\right\rangle_{L^2(\rno)}\\&=-2\sm \lambda_k\int_{0}^\infty\varrho^{n-1}(d_{k}^{\prime})^2\dr - 2(Q-4)\sm\lambda_k\int_{0}^\infty\varrho^{n-3}d_{k}^2\dr.
		\end{align*}
		Hence, we derive the identity
		\begin{align*}
			\norm\frac{\lapg u}{\psi^{1/2}}\norm^2_{L^2(\rno)} & - \quad \norm\frac{\rlapg u}{\psi^{1/2}}\norm^2_{L^2(\rno)}\\&=\sm 8 \lambda_k^2 \int_{0}^\infty \varrho^{n-3}d_{k}^2\dr+4\sm \lambda_k\int_{0}^\infty\varrho^{n-1}(d_{k}^{\prime})^2\dr \\& \qquad+ 4(Q-4)\sm\lambda_k\int_{0}^\infty\varrho^{n-3}d_{k}^2\dr\\&=16\sm \lambda_k^2\norm\frac{d_{k}\Phi_k\psi^{1/2}}{\varrho^2}\norm^2_{L^2(\rno)}+8\sm \lambda_k\norm\frac{d_{k}^\prime\Phi_k\psi^{1/2}}{\varrho}\norm^2_{L^2(\rno)}\\& \qquad +8(Q-4)\sm \lambda_k\norm\frac{d_{k}\Phi_k\psi^{1/2}}{\varrho^2}\norm^2_{L^2(\rno)}.
		\end{align*}
		Now, writing back for the function $u$, we obtain the identity.
	\end{proof}
	
	\begin{remark}
		{\rm
			Let $Q\geq 4$ and $u\in C_c^\infty(\rno\setminus\{o\})$. Then, writing in terms of the given function and using the estimate $\lambda_k\geq (Q-1)/4$, there holds
			\begin{multline*}
				\norm\frac{\lapg u}{\psi^{1/2}}\norm^2_{L^2(\rno)}\geq \quad\norm\frac{\rlapg u}{\psi^{1/2}}\norm^2_{L^2(\rno)}+2(Q-1)\norm\frac{u_\varrho\psi^{1/2}}{\varrho}\norm^2_{L^2(\rno)}\\+3(Q-3)(Q-1)\norm\frac{u\psi^{1/2}}{\varrho^2}\norm^2_{L^2(\rno)}.
			\end{multline*}
		}
	\end{remark}
	\begin{remark}
		{\rm
			After neglecting the non-negative remainder terms in the identity of Theorem~\ref{rellichprojection}, we have the following comparison result. Let $Q\geq 4$ and $u\in C_c^\infty(\rno \setminus \{o\})$, there holds
			\begin{align}\label{rad_lap_comp}
				\int_{\rno}\frac{|\lapg u|^2}{|\gradg \varrho|^2}\dx\dt\geq\int_{\rno}\frac{|\rlapg u|^2}{|\gradg \varrho|^2}\dx\dt.
			\end{align}
			This type of comparison results have been well studied for Riemannian model manifolds whose sectional curvature is bounded by minus one (see \cite[Lemma~6.1]{jmaa}).}
	\end{remark}
	
	\begin{remark}
		{\rm
			The above comparison \eqref{rad_lap_comp}, and neglecting non-negative terms in the identity of part (b) of Corollary \ref{cor-hr-r-rad}, in terms of the radial operator, we can retrieve our known Rellich inequality. For $Q\geq 5$ and $u\in C_c^\infty(\rno \setminus \{o\})$, there holds
			\begin{align*}
				\int_{\rno}\frac{|\lapg u|^2}{|\gradg \varrho|^2}\dx \dt\geq\int_{\rno}\frac{|\rlapg u|^2}{|\gradg \varrho|^2}\dx\dt\geq \frac{Q^2(Q-4)^2}{16}\int_{\rno} \frac{|u|^2|\gradg \varrho|^2}{\varrho^4}\dx\dt.
			\end{align*}
			This gives the improvement of \eqref{rellich} and also settles the optimality issue of the Rellich constant.
		}
	\end{remark}
	
	It is worth mentioning that Theorem~\ref{rellichspherical} and Theorem~\ref{rellichprojection} are equivalent in some sense. Here we show a connection between the spherical derivatives and projection operators mentioned earlier. The lemma reads as follows :
	\begin{lemma}\label{compare-1}
		Let $Q\geq 4$ and $u\in C_c^\infty(\rno\setminus\{o\})$. Then we have
		\begin{itemize}
			\item[(a)] \begin{align*}
				\int_{\rno}\frac{|\sum_{j=1}^{n+1}L_j^2 u|^2}{|\gradg \varrho|^2}\dx \dt=16\sum_{k=0}^\infty\lambda_k^2\norm\frac{(\mathcal{P}_k u)\psi^{1/2}}{\varrho^2}\norm^2_{L^2(\rno)};
			\end{align*}
			\item[(b)] \begin{align*}
				\sum_{j=1}^{n+1}\int_{\rno}\frac{|L_j u|^2}{\varrho^2}\dx\dt=4\sm \lambda_k\norm\frac{(\mathcal{P}_k u)\psi^{1/2}}{\varrho^2}\norm^2_{L^2(\rno)}.
			\end{align*}
		\end{itemize}
	\end{lemma}
	\begin{proof}
		Let us first write the proof of part (a). Begin with the spherical harmonic decomposition of $u$ and we can write 
		$$u(x)=u(\varrho,\sigma)=\sum_{k=0}^{\infty}d_{k}(\varrho)\Phi_k(\sigma)=\sum_{k=0}^{\infty}\mathcal{P}_ku, \quad \text{ and }-\mathcal{L}_{\sigma}\Phi_k(\sigma)=\lambda_k\Phi_k(\sigma).$$
		Now, using \eqref{1}, we have
		\begin{align*}
			\int_{\rno}\frac{|\sum_{j=1}^{n+1}L_j^2 u|^2}{|\gradg \varrho|^2}\dx \dt=16\int_{\rno}\frac{\psi}{\varrho^4}|\mathcal{L}_\sigma u|^2\dx \dt=16\sum_{k=0}^\infty\lambda_k^2\int_{\rno}\frac{\psi}{\varrho^4}(\mathcal{P}_k u)^2\dx\dt.
		\end{align*}
		Using this, we deduce part (a). For part (b), we use the integration by parts formula mentioned in part (4), and then part (5) of Lemma~\ref{iden-lem}. Finally we use \eqref{2}, and \eqref{1}, and obtain 
		\begin{align*}
			\sum_{j=1}^{n+1}\int_{\rno}\frac{|L_j u|^2}{\varrho^2}\dx\dt=-\sum_{j=1}^{n+1}\int_{\rno}\varrho^{-2} u (L_j^2 u)\dx\dt=\sum_{k=0}^\infty4\lambda_k \int_{\rno}\frac{\psi( \mathcal{P}_k u)^2}{\varrho ^4}\dx\dt.
		\end{align*}
	\end{proof}
	
	\begin{lemma}\label{compare-2}
		Let $Q\geq 4$ and $u\in C_c^\infty(\rno\setminus\{o\})$. Then we have
		\begin{multline*}
			\sum_{j=1}^{n+1}\int_{\rno}\varrho^{2-Q}\left|\frac{\partial }{\partial \varrho}\left(L_ju\varrho^{\frac{Q-2}{2}}\right)\right|^2\dx\dt\\=4\sm \lambda_k\norm\frac{(\mathcal{P}_k u_\varrho)\psi^{1/2}}{\varrho}\norm^2_{L^2(\rno)}-(Q-4)^2\sm \lambda_k\norm\frac{(\mathcal{P}_k u)\psi^{1/2}}{\varrho^2}\norm^2_{L^2(\rno)}.
		\end{multline*}	
	\end{lemma}
	\begin{proof}
		By using part (5) of Lemma~\ref{iden-lem} and repeatedly exploiting integration by parts, we have
		\begin{align*}
			&\sum_{j=1}^{n+1}\int_{\rno}\varrho^{2-Q}\left|\frac{\partial }{\partial \varrho}\left(L_ju\varrho^{\frac{Q-2}{2}}\right)\right|^2\dx\dt\\&=\sum_{j=1}^{n+1}\int_{\rno}|(L_j u)_\varrho|^2\dx\dt+\frac{(Q-2)^2}{4}\sum_{j=1}^{n+1}\int_{\rno}\varrho^{-2}(L_j u)^2\dx\dt\\&+(Q-2)\sum_{j=1}^{n+1}\int_{\rno}\varrho^{-1}(L_j u)_\varrho(L_j u)\dx\dt\\&=\sum_{j=1}^{n+1}\int_{\rno}|(L_j u)_\varrho|^2\dx\dt-(Q-2)^2\sm \lambda_k\norm\frac{(\mathcal{P}_k u)\psi^{1/2}}{\varrho^2}\norm^2_{L^2(\rno)}.
		\end{align*}
		Now using \eqref{3}, we have $L_j(u_\varrho)=(L_j u)_\varrho+\varrho^{-1}L_j u$, and using integration by parts we obtain 
		\begin{multline*}
			\sum_{j=1}^{n+1}\int_{\rno}|(L_j u)_\varrho|^2\dx\dt\\=4\sm \lambda_k\norm\frac{(\mathcal{P}_k u_\varrho)\psi^{1/2}}{\varrho}\norm^2_{L^2(\rno)}+4(Q-3)\sm \lambda_k\norm\frac{(\mathcal{P}_k u)\psi^{1/2}}{\varrho^2}\norm^2_{L^2(\rno)}.
		\end{multline*}
		In the above, we use Lemma~\ref{compare-1} twice. Now, substituting this into the above and combining the two terms we deduce the required identity.
	\end{proof}
	
	\begin{remark}
		{\rm
			By utilizing Lemma~\ref{compare-1} and Lemma~\ref{compare-2}, we can compare all terms in Theorem~\ref{rellichspherical} and Theorem~\ref{rellichprojection}. After combining the terms, we find that the Rellich identity, expressed in terms of spherical derivatives and projection operators, is the same. 
		}
	\end{remark}
	
	\medspace
	
	\section{Symmetrization principle on Grushin space}	
	This section aims to discuss the symmetrization type principle in the Grushin space. The rearrangement type arguments, particularly the P\'olya-Szeg\"o inequality are not true when dealing with the higher-order operators, even for the Euclidean Laplacian. Here, we shall discuss the symmetrization principle in the Grushin space. The results obtained here can be compared with \cite[Theorem~1.11]{dll} in the Euclidean setting and we also refer to \cite{GM1} for the related open problems. The space $C_{c, rad}^\infty$ denotes the space of all radial functions. Let us begin with the following first-order inequality.
	\begin{theorem}
		Let $G$ be the Grushin space with dimension $Q\geq 4$ and $B_R(o)$ denotes the $\varrho-$gauge ball with radius $R$ for $0<R\leq \infty$. Assume $W$ be any function on $(0,R)$. Then we have
		\begin{align*}
			\int_{B_{R}(o)} |\gradg u|^{2} \dx\dt \geq 
			\int_{B_{R}(o)} W(\varrho)|u|^{2}|\gradg \varrho|^2\dx\dt \quad \text{ for all } u\in C_{c}^{\infty}\left(B_{R}(o) \backslash\{o\}\right)
		\end{align*}
		if and only if
		\begin{align*}
			\int_{B_{R}(o)} |\gradg u|^{2} \dx\dt \geq 
			\int_{B_{R}(o)} W(\varrho)|u|^{2}|\gradg \varrho|^2\dx\dt \quad \text{ for all } u\in C_{c,rad}^{\infty}\left(B_{R}(o) \backslash\{o\}\right).
		\end{align*}
	\end{theorem}
	\begin{proof}
		The one-way implication is trivial. So it is enough to show the other way. Let us begin with the spherical harmonic decomposition of $u$ and we can write
		\begin{align*}
			u(\varrho,\sigma)=\sm d_k(\varrho)\Phi_k(\sigma).
		\end{align*}
		Next, exploiting the polar coordinate structure, we deduce
		\begin{align}\label{sym-hardy-eq-1}
			\int_{B_R(o)} |\gradg u|^2\dx\dt
			=\frac{1}{2}\sm\int_{0}^R\bigg[{d_k^\prime}^2\varrho^{n+1}+4\lambda_k  d_k^2\varrho^{n-1}\bigg]\dr,
		\end{align}
		and
		\begin{align}\label{sym-hardy-eq-2}
			\int_{B_R(o)}W(\varrho)|u|^2|\gradg \varrho|^2\dx\dt=\frac{1}{2}\sm\izfr W(\varrho)d_k^2\varrho^{n+1}\dr.
		\end{align}
		Now, using the result for each radial function $d_k(\varrho)$, we have
		\begin{align}\label{sym-hardy-eq-3}
			\int_{0}^R{d_k^\prime}^2\varrho^{n+1}\dr \geq \izfr W(\varrho)d_k^2\varrho^{n+1}\dr.
		\end{align}
		After subtracting \eqref{sym-hardy-eq-2} from \eqref{sym-hardy-eq-1}, and then using \eqref{sym-hardy-eq-3}, it only remains to check for each $k$, the non-negativity of the following term $4\lambda_k  \izfr d_k^2\varrho^{n-1}\dr$. This is immediate as for each $k\geq 1$, we have $4\lambda_{k}\geq (Q-1)$. This completes the proof.
	\end{proof}
	
	Now we are in a situation to prove the second-order version of the above theorem.  We obtain the following symmetrization principle for the second-order operator : 
	\begin{theorem}
		Let $G$ be the Grushin space with dimension $Q\geq 4$ and $B_R(o)$ denotes the $\varrho-$gauge ball with radius $R$ for $0<R\leq \infty$. Assume $W$ be any function on $(0,R)$. Then we have 
		\begin{align*}
			\int_{B_{R}(o)} \frac{(\lapg u)^{2}}{|\gradg \varrho|^{2}} \dx\dt \geq 
			\int_{B_{R}(o)} W(\varrho)|u|^{2}|\gradg \varrho|^2\dx\dt \quad \text{ for all } u\in C_{c}^{\infty}\left(B_{R}(o) \backslash\{o\}\right)
		\end{align*}
		if and only if
		\begin{align*}
			\int_{B_{R}(o)} \frac{(\lapg u)^{2}}{|\gradg \varrho|^{2}} \dx\dt \geq 
			\int_{B_{R}(o)} W(\varrho)|u|^{2}|\gradg \varrho|^2\dx\dt \quad \text{ for all } u\in C_{c,rad}^{\infty}\left(B_{R}(o) \backslash\{o\}\right).
		\end{align*}
	\end{theorem}
	\begin{proof}
		The one-sided implication follows trivially. So it is enough to show the other way. Let us begin with the spherical harmonic decomposition of $u$ and we can write
		\begin{align*}
			u(\varrho,\sigma)=\sm d_k(\varrho)\Phi_k(\sigma).
		\end{align*}
		By exploiting the polar coordinate structure, we deduce
		\begin{align}\label{sym-rell-eq-1}
			& \int_{B_{R}(o)} \frac{(\lapg u)^{2}}{|\gradg \varrho|^{2}} \dx\dt =\frac{1}{2}  \sum_{k = 0}^{\infty} \bigg[\int_0^R  \left( d_{k}^{\prime \prime} + \frac{(Q-1)}{\varrho}d_{k}^{\prime} \right)^2 \varrho^{n+1}\dr \nonumber \\& + 
			16 \lambda_{k}^2\int_0^R d_{k}^2 \varrho^{n-3} \dr - 8\lambda_k \int_0^R  \left( d_{k}^{\prime \prime} + \frac{(Q-1)}{\varrho}d_{k}^{\prime} \right) d_k \varrho^{n-1}\dr\bigg],
		\end{align}
		and
		\begin{align}\label{sym-rell-eq-2}
			\int_{B_R(o)}W(\varrho)|u|^2|\gradg \varrho|^2\dx\dt=\frac{1}{2}\sm\izfr W(\varrho)d_k^2\varrho^{n+1}\dr.
		\end{align}
		Then, using the result for each radial function $d_k(\varrho)$, we have
		\begin{align}\label{sym-rell-eq-3}
			\int_0^R  \left( d_{k}^{\prime \prime} + \frac{(Q-1)}{\varrho}d_{k}^{\prime} \right)^2 \varrho^{n+1}\dr \geq \izfr W(\varrho)d_k^2\varrho^{n+1}\dr.
		\end{align}
		After subtracting \eqref{sym-rell-eq-2} from \eqref{sym-rell-eq-1}, and then using \eqref{sym-rell-eq-3}, it only remains to check for each $k$, the non-negativity of the following term 
		\begin{align*}
			\mathcal{M}:=16 \lambda_{k}^2\int_0^R d_{k}^2 \varrho^{n-3} \dr - 8\lambda_k \int_0^R   d_{k}^{\prime \prime} d_k \varrho^{n-1}\dr-8\lambda_{k}(Q-1)\izfr d_{k}^{\prime}  d_k \varrho^{n-2}\dr.
		\end{align*} 
		Now, using the integration by parts we have
		\begin{align}\label{bp-1-1}
			\izfr d_{k}^{\prime}  d_k \varrho^{n-2}\dr=-\frac{(Q-4)}{2}\int_0^R d_{k}^{2}  \varrho^{n-3}\dr,	
		\end{align}
		and
		\begin{align}\label{bp-1-2}
			\int_0^R   d_{k}^{\prime \prime} d_k \varrho^{n-1}\dr&=-\int_{0}^R {d_k^\prime}^2\varrho^{n-1}\dr-(Q-3)\int_0^R  d_{k}^{\prime}  d_k \varrho^{n-2}\dr\nonumber\\&=-\int_{0}^R {d_k^\prime}^2\varrho^{n-1}\dr+\frac{(Q-3)(Q-4)}{2}\int_0^R   d_k^2 \varrho^{n-3}\dr.
		\end{align}
		Using \eqref{bp-1-1} and \eqref{bp-1-2} in $\mathcal{M}$, we obtain
		\begin{align*}
			\mathcal{M}&=8\lambda_{k}\int_{0}^R {d_k^\prime}^2\varrho^{n-1}\dr+8\lambda_{k}\big(2 \lambda_{k}+Q-4\big)\int_0^R d_{k}^2 \varrho^{n-3} \dr.
		\end{align*}
		This follows as for each $k\geq 1$, we have $2 \lambda_{k}+Q-4\geq \frac{3}{2}(Q-3)$. Finally, using $Q\geq 4$, we obtain $\mathcal{M}$ is non-negative and the result follows.
	\end{proof}
	
	Our next theorem deals with the Hardy-Rellich type symmetrization principle. It is worth noting that while proving the Hardy-Rellich inequality, we need to make an additional assumption on the function $W.$ The inequality reads as follows : 
	\begin{theorem}\label{symmrellich}
		Let $G$ be the Grushin space with dimension $Q\geq 4$ and $B_R(o)$ denotes the $\varrho-$gauge ball with radius $R$ for $0<R\leq \infty$. Assume $W(t)$ be any function on $(0,R)$ satisfying 
		\begin{align}\label{sym-hardy-rell-cond}
			2\int_{0}^R {f^\prime}^2(t)t^{Q-3}\dt&+2\big(2 \lambda_{1}+Q-4\big)\int_0^R f^2(t) t^{Q-5} \dt \nonumber \\&  \geq \izfr W(t)f^2(t)t^{Q-3}\dt \quad \text{ for all } f\in C_c^\infty(0,R),
		\end{align}
		where $\lambda_{1}=(Q-1)/4$. Then we have 
		\begin{align*}
			\int_{B_{R}(o)} \frac{(\lapg u)^{2}}{|\gradg \varrho|^{2}} \dx\dt \geq 
			\int_{B_{R}(o)} W(\varrho)|\gradg u|^{2}\dx\dt \quad \text{ for all } u\in C_{c}^{\infty}\left(B_{R}(o) \backslash\{o\}\right)
		\end{align*}
		if and only if
		\begin{align*}
			\int_{B_{R}(o)} \frac{(\lapg u)^{2}}{|\gradg \varrho|^{2}} \dx\dt \geq 
			\int_{B_{R}(o)} W(\varrho)|\gradg u|^{2}\dx\dt \quad \text{ for all } u\in C_{c,rad}^{\infty}\left(B_{R}(o) \backslash\{o\}\right).
		\end{align*}
	\end{theorem}
	\begin{proof}
		The one-sided implication follows trivially. So it is enough to show the other way. Let us begin with the spherical harmonic decomposition of $u$ and we can write
		\begin{align*}
			u(\varrho,\sigma)=\sm d_k(\varrho)\Phi_k(\sigma).
		\end{align*}
		Next, exploiting the polar coordinate structure we deduce
		\begin{align}\label{sym-hardy-rell-eq-1}
			& \int_{B_{R}(o)} \frac{(\lapg u)^{2}}{|\gradg \varrho|^{2}} \dx\dt =\frac{1}{2}  \sum_{k = 0}^{\infty} \bigg[\int_0^R  \left( d_{k}^{\prime \prime} + \frac{(Q-1)}{\varrho}d_{k}^{\prime} \right)^2 \varrho^{n+1}\dr \nonumber \\& + 
			16 \lambda_{k}^2\int_0^R d_{k}^2 \varrho^{n-3} \dr - 8\lambda_k \int_0^R  \left( d_{k}^{\prime \prime} + \frac{(Q-1)}{\varrho}d_{k}^{\prime} \right) d_k \varrho^{n-1}\dr\bigg],
		\end{align}
		and
		\begin{align}\label{sym-hardy-rell-eq-2}
			\int_{B_R(o)}W(\varrho) |\gradg u|^2\dx\dt
			=\frac{1}{2}\sm\int_{0}^R W(\varrho) \bigg[{d_k^\prime}^2\varrho^{n+1}+4\lambda_k  d_k^2\varrho^{n-1}\bigg]\dr.
		\end{align}
		Then, using the result for each radial function $d_k(\varrho)$, we have
		\begin{align}\label{sym-hardy-rell-eq-3}
			\int_0^R  \left( d_{k}^{\prime \prime} + \frac{(Q-1)}{\varrho}d_{k}^{\prime} \right)^2 \varrho^{n+1}\dr \geq \izfr W(\varrho) {d_k^\prime}^2\varrho^{n+1}\dr.
		\end{align}
		After subtracting \eqref{sym-hardy-rell-eq-2} from \eqref{sym-hardy-rell-eq-1}, and then using \eqref{sym-hardy-rell-eq-3}, it is only remains to check for each $k\geq 1$, the non-negativity of the following term 
		\begin{align*}
			\mathcal{B}_k&:=4 \lambda_{k}\int_0^R d_{k}^2 \varrho^{n-3} \dr - 2 \int_0^R   d_{k}^{\prime \prime} d_k \varrho^{n-1}\dr\\&-2(Q-1)\izfr d_{k}^{\prime}  d_k \varrho^{n-2}\dr-  \izfr W(\varrho)d_k^2\varrho^{n-1}\dr.
		\end{align*} 
		Using \eqref{bp-1-1} and \eqref{bp-1-2} in $\mathcal{B}_k$, we obtain
		\begin{align*}
			\mathcal{B}_k=	2\int_{0}^R {d_k^\prime}^2\varrho^{n-1}\dr+2\big(2 \lambda_{k}+Q-4\big)\int_0^R d_{k}^2 \varrho^{n-3} \dr -  \izfr W(\varrho)d_k^2\varrho^{n-1}\dr
		\end{align*}
		We will be done if we prove $\mathcal{B}_k\geq 0$ for all $k\geq 2$, as the $k=1$ case is already done if we choose the radial function $d_1$ in \eqref{sym-hardy-rell-cond}. Now define 
		\begin{align*}
			\mathcal{B}_{1,k}:=	2\int_{0}^R {d_k^\prime}^2\varrho^{n-1}\dr+2\big(2 \lambda_{1}+Q-4\big)\int_0^R d_{k}^2 \varrho^{n-3} \dr -  \izfr W(\varrho)d_k^2\varrho^{n-1}\dr
		\end{align*}
		for each $k\geq 2$. Then it is easy to observe that by choosing $d_k$ in \eqref{sym-hardy-rell-cond}, we have $\mathcal{B}_{1,k}\geq 0$ for each $k\geq 2$. Now we notice
		\begin{align*}
			\mathcal{B}_k-\mathcal{B}_{1,k}&= 2\int_{0}^R {d_k^\prime}^2\varrho^{n-1}\dr+2\big(2 \lambda_{k}+Q-4\big)\int_0^R d_{k}^2 \varrho^{n-3} \dr -  \izfr W(\varrho)d_k^2\varrho^{n-1}\dr \\&- 2\int_{0}^R {d_k^\prime}^2\varrho^{n-1}\dr-2\big(2 \lambda_{1}+Q-4\big)\int_0^R d_{k}^2 \varrho^{n-3} \dr +  \izfr W(\varrho)d_k^2\varrho^{n-1}\dr\\&=4\big( \lambda_{k}-\lambda_{1}\big)\int_0^R d_{k}^2 \varrho^{n-3} \dr.
		\end{align*}
		Finally, observing for each $k\geq 2$, we have $4(\lambda_{k}-\lambda_{1})\geq (Q^2-3Q+1)\geq 0$, when $Q\geq 4$, we say each $\mathcal{B}_k$ is nonnegative. This completes the proof.
	\end{proof}
	
	\begin{remark}
		{\rm
			The rearrangement type arguments are valid in Hardy-Rellich type inequalities for Grushin spaces with every homogeneous dimension $Q\geq 4$ under an additional assumption \eqref{sym-hardy-rell-cond}. It is worth noting that such an assumption has far-reaching consequences in the Euclidean setting and is related to the symmetry-breaking phenomenon depending on the dimension of the ambient space. For further details, we refer to \cite[Remark 1.2]{dll}.
		}
	\end{remark}
	
	\medspace
	
	\section{Applications: Second-order uncertainty Principles and Caffarelli-Kohn-Nirenberg inequalities}
	This final section is devoted to the application of the results we obtained in the previous sections. Recently, there has been a growing interest in obtaining sharp uncertainty principles and the existence of extremals. To begin with, we shall observe some immediate facts of our main theorems. One of the important aspects of Corollary \ref{sym-hardy} is that the Hardy inequality holds for both radial and non-radial functions. However, we observed that in Theorem \ref{nr-rell} we require an additional condition \eqref{cond} to obtain the Hardy-Rellich type inequality. In this part, we shall examine the outcomes when the dimensions of the Grushin space and Bessel pair differ. The following result can be compared with \cite[Theorem 1.3, Theorem 1.4]{dll}.  It is important to note that we will obtain the identity without any sign assumption on $V,$ where $V$ is a radial function, while for the non-radial setting, we need a non-negative sign assumption on $V.$ The result is as follows :
	\begin{proposition}\label{r-nr-dim-rell}
		Let $G$ be the Grushin space with dimension $Q\geq 4$ and $B_R(o)$ denotes the $\varrho-$gauge ball with radius $R$ for $0<R\leq \infty$. Assume that $\big(V(r), W(r)+Q\frac{V_r(r)}{r}\big)$ is a $(Q+2)$-dimensional Bessel pair on $(0,R)$ with positive solution $f(r)$. Then for all radial function $u \in C_{c}^{\infty}\left(B_{R}(o) \backslash\{o\}\right)$ there holds
		\begin{align*}
			\int_{B_{R}(o)} V(\varrho) \frac{(\lapg u)^{2}}{|\gradg \varrho|^{2}} \dx\dt  &=\int_{B_{R}(o)} W(\varrho)|\gradg u|^{2}\dx\dt\\&
			+\int_{B_{R}(o)} V(\varrho)\left|\nabla_{G}\left(\frac{u_\varrho}{\varrho f(\varrho)}\right)\right|^{2}\varrho^2 f^{2}(\varrho)\dx\dt.
		\end{align*}
		Moreover, if $V\geq 0$ and it satisfies \eqref{cond}, then for all $u \in C_{c}^{\infty}\left(B_{R}(o) \backslash\{o\}\right)$ there holds
		\begin{align*}
			\int_{B_{R}(o)} V(\varrho) \frac{(\lapg u)^{2}}{|\gradg \varrho|^{2}} \dx\dt & \geq\int_{B_{R}(o)} W(\varrho)|\gradg u|^{2}\dx\dt \\
			& +\int_{B_{R}(o)} V(\varrho)\left|\nabla_{G}\left(\frac{u_\varrho}{\varrho f(\varrho)}\right)\right|^{2} \varrho^2 f^{2}(\varrho)\dx\dt,
		\end{align*}
		where the $\varrho$ in the suffix means the derivative with respect to the radial part $\varrho$.
	\end{proposition}
	\begin{proof}
		Notice that if $\big(V(r), W(r)+Q\frac{V_r(r)}{r}\big)$ is a $(Q+2)$-dimensional Bessel pair with positive solution $f(r)$, then $\bigg(V(r), W(r)-(Q-1)\biggl\{\frac{V(r)}{r^2}-\frac{V_r(r)}{r}\biggr\}\bigg)$ is a $Q$-dimensional Bessel pair on $(0,R)$ with positive solution $rf(r)$. Now applying Theorem \ref{r-rell} and Theorem \ref{nr-rell}, we deduce our results.
	\end{proof}
	
	As a consequence of this result, we obtain the following double-weighted Hardy-Rellich inequality. 	
	\begin{corollary}
		Let $G$ be the Grushin space with dimension $Q\geq 5$ and $B_R(o)$ denotes the $\varrho-$gauge ball with radius $R$ for $R>0$.  Then for all $u\in C_c^\infty(B_R(o) \setminus \{o\})$, we have
		\begin{align*}
			\int_{B_R(o)}\frac{|\mathcal{L}_{G} u|^2}{|\gradg \varrho|^2}\dx \dt\geq \frac{Q^2}{4}\int_{B_R(o)} \frac{|\nabla_{G} u|^2}{|\varrho|^2\big(1-\big(\frac{R}{\varrho}\big)^{-Q}\big)^2}\dx\dt.
		\end{align*}
	\end{corollary}
	\begin{proof}
		Applying $\left(1,\frac{Q^2}{4}r^{-2}\left(1-\left(\frac{r}{R}\right)^{Q}\right)^{-2}\right)$ as a $(Q+2)$ dimensional Bessel pair in Proposition~\ref{r-nr-dim-rell} and then removing the non-negative remainder term, we obtain the result.
	\end{proof}
	
	\begin{remark}
		{\rm
			Recently, the above inequality has been obtained in the Cartan-Hadamard manifolds and nilpotent Lie groups in  \cite{RSY} and  \cite{SY-arxiv} respectively. 
			
		}
	\end{remark}	
	
	We can also deduce several Caffarelli-Kohn-Nirenberg-type inequalities. Recently, the sharp $L^2$-Caffarelli-Kohn-Nirenberg inequalities for Grushin vector fields were investigated in \cite{flynn}, and in \cite{cfl} the sharp second-order uncertainty principles were studied in the Euclidean space. Now we establish a few second-order uncertainty principles (USP) and among them, the Heisenberg and Hydrogen USP are the two main results.
	\begin{corollary}\label{ckn-cor}
		Let $Q\geq 5$. For all $u\in C_c^\infty(\rno \setminus \{o\})$, we have
		\begin{itemize}
			\item[(a)]
			\begin{align*}
				\int_{\rno}\frac{|\mathcal{L}_{G} u|^2}{|\gradg \varrho|^2}\dx \dt\geq \int_{\rno} (Q+2-\varrho^2)|\gradg u|^2\dx\dt.
			\end{align*}
			\item[(b)]
			\begin{align*}
				\int_{\rno}\frac{|\mathcal{L}_{G} u|^2}{|\gradg \varrho|^2}\dx \dt\geq \int_{\rno}\bigg(\frac{Q+1}{\varrho}-1\bigg)|\gradg u|^2\dx\dt.
			\end{align*}
			\item[(c)] for $b\neq 1:$
			\begin{align*}
				\int_{\rno}\frac{|\mathcal{L}_{G} u|^2}{|\gradg \varrho|^2}\dx \dt\geq\int_{\rno}\bigg(\frac{Q+1-b}{\varrho^{b+1}}-\frac{1}{\varrho^{2b}}\bigg)|\gradg u|^2\dx\dt.
			\end{align*}
		\end{itemize}
	\end{corollary}
	\begin{proof}
		We apply the second part of Proposition~\ref{r-nr-dim-rell} with appropriate $(Q+2)$-dimensional Bessel pairs and after deducting the remainder terms we obtain the results. Here are the respective Bessel pairs: $(1,Q+2-r^2)$, $\big(1,\frac{Q+1}{r}-1\big)$, and $\big(1,\frac{Q+1-b}{r^{b+1}}-\frac{1}{r^{2b}}\big)$ for $b\neq 1$.
	\end{proof}
	We can also deduce the following second-order Caffarelli-Kohn-Nirenberg (CKN) inequalities on Grushin space from the above Corollary \ref{ckn-cor}. Also, see \cite{HZ, SL} for certain types of CKN inequalities. 
	\begin{corollary}
		Let $Q\geq 5$. For all $u\in C_c^\infty(\rno \setminus \{o\})$, we have
		\begin{itemize}
			\item[(a)] $(\text{Second order Heisenberg Uncertainty Principle})$
			\begin{align*}
				\bigg(\int_{\rno}\frac{|\mathcal{L}_{G} u|^2}{|\gradg \varrho|^2}\dx \dt\bigg)^{\frac{1}{2}}\bigg(\int_{\rno}\varrho^2|\gradg u|^2\dx\dt\bigg)^{\frac{1}{2}}\geq \frac{(Q+2)}{2}\int_{\rno} |\gradg u|^2\dx\dt.
			\end{align*}
			Moreover, the sharp constant is $\frac{(Q+2)}{2}$, and achieved by the functions of the form $u(x,t)=\alpha e^{-\beta\varrho^2(x,t)}$, $\alpha\in \mathbb{R}$, $\beta>0$.
			\item[(b)] $(\text{Second order Hydrogen Uncertainty Principle})$
			\begin{align*}
				\bigg(\int_{\rno}\frac{|\mathcal{L}_{G} u|^2}{|\gradg \varrho|^2}\dx \dt\bigg)^{\frac{1}{2}}\bigg(\int_{\rno}|\gradg u|^2\dx\dt\bigg)^{\frac{1}{2}}\geq \frac{(Q+1)}{2}\int_{\rno} \frac{|\gradg u|^2}{\varrho}\dx\dt.
			\end{align*}
			Moreover, the sharp constant is $\frac{(Q+1)}{2}$, and achieved by the functions of the form $u(x,t)=\alpha(1+\beta\varrho(x,t)) e^{-\beta\varrho(x,t)}$, $\alpha\in \mathbb{R}$, $\beta>0$.
			\item[(c)] for $b< 1:$
			\begin{align*}
				\bigg(\int_{\rno}\frac{|\mathcal{L}_{G} u|^2}{|\gradg \varrho|^2}\dx \dt\bigg)^{\frac{1}{2}}\bigg(\int_{\rno}\frac{|\gradg u|^2}{\varrho^{2b}}\dx\dt\bigg)^{\frac{1}{2}}\geq \frac{(Q+1-b)}{2}\int_{\rno} \frac{|\gradg u|^2}{\varrho^{b+1}}\dx\dt.
			\end{align*}
			Moreover, the sharp constant is $\frac{(Q+1-b)}{2}$, and achieved by the functions of the form $u(x,t)=\alpha\int_{\varrho(x,t)}^\infty s\exp\left(-\beta\frac{s^{1-b}}{1-b}\right){\rm d}s$, $\alpha\in \mathbb{R}$, $\beta>0$.
		\end{itemize}
	\end{corollary}
	\begin{proof}
		Using the standard scaling argument the proof follows. For clarity, we will write a detailed proof here. We will only write the proof of part (a), and the others will follow similarly. Let us write the dilation for positive real number $\lambda$, on Grushin space by $\delta_\lambda(x,t):=(\lambda x, \lambda^2 t)$, where $(x,t) \in \rno\setminus\{o\}$. Now for $u\in C_c^\infty(\rno \setminus \{o\})$, let us define the scaling of the function by $u_\lambda(x,t):=\lambda^{\frac{Q-2}{2}}u(\delta_\lambda(x,t))$, and it is easy to notice that $u_\lambda\in C_c^\infty(\rno \setminus \{o\})$. Now, consider the change of variables $\lambda x=z$ and $\lambda^2 t= \mu$ and we are denoting the Grushin gradient and Laplacian operator as $\nabla_{g}$ and $\mathcal{L}_g$ in terms of the $(z,\mu)$ variables. Then we have the following
		\begin{align*}
			\dx\dt = \lambda^{-Q}{\rm d}z{\rm d}\mu,
		\end{align*}
		\begin{align*}
			\varrho^2(x,t)=\big(|x|^4+4t^2\big)^{\frac{1}{2}}=\lambda^{-2}\varrho^2(z,\mu),
		\end{align*}
		and
		\begin{align*}
			|\gradg\varrho(x,t)|^2=\frac{|x|^2}{\varrho^2(x,t)}=\frac{\lambda^{-2}|z|^2}{\lambda^{-2}\varrho^2(z,\mu)}= |\nabla_{g} \varrho (z,\mu)|^2.
		\end{align*}
		Moreover, we observe that
		\begin{align*}
			|\gradg u_\lambda(x,t)|^2=\lambda^Q|\nabla_{g} u(z,\mu)|^2,	
		\end{align*}
		and
		\begin{align*}
			\lapg u_\lambda(x,t)= \lambda^{\frac{Q-2}{2}} \lambda^2 \mathcal{L}_g u(z,\mu).
		\end{align*}
		
		Applying Corollary \ref{ckn-cor}, part (a), for $u_\lambda\in C_c^\infty(\rno \setminus \{o\})$, with $\lambda>0$, we have
		\begin{align*}
			\int_{\rno}\frac{|\mathcal{L}_{G} u_\lambda|^2}{|\gradg \varrho|^2}\dx \dt+\int_{\rno} \varrho^2|\gradg u_\lambda|^2\dx\dt\geq (Q+2)\int_{\rno} |\gradg u_\lambda|^2\dx\dt.
		\end{align*}
		Then writing in terms of $u$ and performing the above change of variable we deduce
		\begin{align*}
			\lambda^{2}\int_{\rno}\frac{|\mathcal{L}_{g} u|^2}{|\nabla_g \varrho|^2}\,{\rm d}z{\rm d}\mu+\lambda^{-2}\int_{\rno} \varrho^2|\nabla_g u|^2\,{\rm d}z{\rm d}\mu\geq (Q+2)\int_{\rno} |\nabla_g u|^2\,{\rm d}z{\rm d}\mu.
		\end{align*}
		Finally, choose
		\begin{align*}
			\lambda=\left(\frac{\int_{\rno} \varrho^2|\nabla_g u|^2\,{\rm d}z{\rm d}\mu}{\int_{\rno}\frac{|\mathcal{L}_{g} u|^2}{|\nabla_g \varrho|^2}\,{\rm d}z{\rm d}\mu}\right)^{\frac{1}{4}},
		\end{align*}
		and we get the desired result. 
		
		Now let us discuss the optimality for part (c). Consider the smooth function 
		\begin{align*}
			u_{\alpha,\beta}(x,t)=\alpha\int_{\varrho(x,t)}^\infty s\exp\left(-\beta\frac{s^{1-b}}{1-b}\right){\rm d}s,
		\end{align*}
		for $\alpha\in \mathbb{R}$, $\beta>0$, and $1-b>0$.  We can easily notice that
		\begin{align*}
			\frac{\partial u_{\alpha,\beta}}{\partial \varrho}= -\alpha \varrho \exp\left(-\beta\frac{\varrho^{1-b}}{1-b}\right),
		\end{align*}
		and
		\begin{align*}
			\frac{\partial^2 u_{\alpha,\beta}}{\partial \varrho^2}=	-\alpha  \exp\left(-\beta\frac{\varrho^{1-b}}{1-b}\right)+\alpha\beta \varrho^{1-b} \exp\left(-\beta\frac{\varrho^{1-b}}{1-b}\right).
		\end{align*}
		
		Now using polar coordinate decomposition, change of variable, and properties of Gamma functions, we have
		\begin{align*}
			\int_{\rno} \frac{|\gradg u_{\alpha,\beta}|^2}{\varrho^{b+1}}\dx\dt=\frac{1}{2}|\Omega|\frac{\alpha^2}{2\beta}\left(\frac{1-b}{2\beta}\right)^{\frac{Q}{1-b}}\left(\frac{Q}{1-b}\right)\Gamma\left(\frac{Q}{1-b}\right),
		\end{align*}
		\begin{align*}
			\int_{\rno}\frac{|\gradg u_{\alpha,\beta}|^2}{\varrho^{2b}}\dx\dt=\frac{1}{2}|\Omega|\frac{\alpha^2}{2\beta}\left(\frac{1-b}{2\beta}\right)^{\frac{Q}{1-b}+1}\left(\frac{Q+1-b}{1-b}\right)\left(\frac{Q}{1-b}\right)\Gamma\left(\frac{Q}{1-b}\right),
		\end{align*}
		and
		\begin{align*}
			\int_{\rno}\frac{|\mathcal{L}_{G} u_{\alpha,\beta}|^2}{|\gradg \varrho|^2}\dx \dt=\frac{1}{2}|\Omega|\frac{\alpha^2Q}{2\beta}\left(\frac{1-b}{2\beta}\right)^{\frac{Q}{1-b}-1}\left(\frac{Q+1-b}{4}\right)\Gamma\left(\frac{Q}{1-b}\right).
		\end{align*}
		Notice that as all the terms of the inequality are finite, hence inequality holds for such functions, which can be justified by the density arguments of smooth, compactly supported functions. Therefore, we have 
		
		\begin{align*}
			\frac{\left(\int_{\rno}\frac{|\mathcal{L}_{G} u_{\alpha,\beta}|^2}{|\gradg \varrho|^2}\dx \dt\right)\left(\int_{\rno}\frac{|\gradg u_{\alpha,\beta}|^2}{\varrho^{2b}}\dx\dt\right)}{\left(\int_{\rno} \frac{|\gradg u_{\alpha,\beta}|^2}{\varrho^{b+1}}\dx\dt\right)^2}=\left(\frac{Q+1-b}{2}\right)^2.
		\end{align*}
		The proof of the optimality of the constant in part (c) has been concluded. The other cases can be proved similarly by considering the corresponding optimizing functions. It should be noted that for parts (a) and (b), the optimizing function is a special case of part (c) with $b=-1$ and $b=0$, respectively.
	\end{proof}
	
	\begin{remark}
		{\rm
			The cases (a) and (b) above of the above theorem can be compared to  \cite[Theorem 2.1, 2.4]{cfl}. Moreover, the inequalities obtained in part (c) are related to a special case of \cite[Theorem 1.3]{newwgckn}. It's worth mentioning that by exploiting the first part of Proposition \ref{r-nr-dim-rell} and the corresponding Bessel pair, we can obtain the results for radial operators as well. These second-order $L^2$-Caffarelli-Kohn-Nirenberg inequalities are a completely new addition to the literature for Grushin vector fields, and provide a natural extension of the first-order case presented in \cite{flynn}. Also, a certain uncertainty principle related to the Rellich inequality was studied in \cite{IK2}.
		}
	\end{remark}

\end{document}